%% file: sharp_bounds.tex
\documentclass[review,onefignum,onetabnum]{siamart171218}

\input{ex_shared}
\nolinenumbers

\ifpdf
\hypersetup{
  pdftitle={Sharp error bounds for turning point expansions},
  pdfauthor={T. M. Dunster, A. Gil, and J. Segura}
}
\fi

\begin{document}

\maketitle

\begin{abstract}
  Computable and sharp error bounds are derived for asymptotic expansions for linear differential equations having a simple turning point. The expansions involve Airy functions and slowly varying coefficient functions. The sharpness of the bounds is illustrated numerically with an application to Bessel functions of large order.
\end{abstract}

\begin{keywords}
  {Asymptotic expansions, Airy functions, Turning point theory, WKB methods}
\end{keywords}

\begin{AMS}
  34E05, 33C10, 34E20
\end{AMS}

\section{Introduction} 
\label{sec1}
In this paper we obtain sharp error bounds for a recent form of asymptotic expansions involving Airy functions and slowly varying coefficient functions for linear differential equations having a simple turning point. This is a sequel to the paper \cite{Dunster:2020:SEB} in which the aforementioned error bounds were obtained in terms of elementary functions. We show that by a manipulation of these bounds we obtain new bounds which are extremely close the the exact errors. The method is based on the one used for bounds derived in \cite{Dunster:2020:LGE} for the Liouville-Green (LG) approximation.

The differential equations we study are of the form 
\begin{equation} 
d^{2}w/dz^{2}=\left\{ {u^{2}f(z) +g(z) }\right\} w, 
\label{eq1} 
\end{equation}
where $u$ is a large parameter, real or complex, and $z$ lies in a complex domain which may be unbounded. The functions $f(z)$ and $g(z)$ are meromorphic in a certain domain $Z$ (precisely defined below), and are independent of $u$ (although the latter restriction can often be relaxed without undue difficulty). We further assume that $f(z)$ has no zeros in $Z$ except for a simple zero at $z=z_{0}$, which is the turning point of the equation. 

In the ensuing asymptotic expansions the following two variables have a prominent role, these being given by 
\begin{equation} 
\xi =\tfrac{2}{3}\zeta ^{3/2}=\pm \int_{z_{0}}^{z}{f^{1/2}(t) dt} . 
\label{xi} 
\end{equation}
The variable $\zeta$ appears in the Airy function expansions which are valid at the turning point, and is an analytic function of $z$ at $z=z_{0}$. As a function of $z$ the Liouville-Green variable $\xi$ has a branch point at the turning point. Any branch in (\ref{xi}) can be chosen provided that $\xi$ is continuous on the paths of integration in the error bounds.

Following \cite{Dunster:2020:SEB} we define three sectors in the $\zeta$ plane by
\begin{equation}
\mathrm{\mathbf{T}}_{j}=\left\{ \zeta :\left\vert {\arg \left( {u^{2/3}\zeta e^{-2\pi ij/3}}\right) }\right\vert \leq {\tfrac{1}{3}}\pi 
\right\} \ \left( {j=0,\pm 1}\right).
\label{oeq6}
\end{equation}

We further partition each of the sectors by defining $\mathrm{\mathbf{T}}_{j}=\mathrm{\mathbf{T}}
_{j,k}\cup \mathrm{\mathbf{T}}_{j,l}$ ($j,k,l\in \left\{0,1,-1\right\}$, $j\neq k\neq l\neq j$), where $\mathrm{\mathbf{T}}_{j,k}$ is the closed
subsector of angle $\pi /3$ and adjacent to $\mathrm{\mathbf{T}}_{k}$; for
example $\mathrm{\mathbf{T}}_{0,1}=\left\{ \zeta {:0\leq \arg \left( {u^{2/3}
}\zeta \right) \leq {\tfrac{1}{3}}\pi }\right\} $. We denote $T_{j}$ (respectively $T_{j,k}$) to be the region in the $z$ plane corresponding to the sector $\mathrm{\mathbf{T}}_{j}$
(respectively $\mathrm{\mathbf{T}}_{j,k}$) in the $\zeta $ plane. See \Cref{fig:fig1} for some typical regions in the right half $z$ plane for the case $z_{0}$ and $u$ positive.

Next, let $Z$ be the $z$ domain containing $z=z_{0}$ in which $f(z)$ has no other zeros, and in which $f(z)$ and $g(z)$ are meromorphic, with poles (if any) at finite points $z=w_{j}$ ($j=1,2,3,\cdots $) such that

(i) $f(z)$ has a pole of order $m>2$, and $g(z)$ is analytic or has a pole of order less than $\frac{1}{2}m+1$, or

(ii) $f(z)$ and $g(z)$ have a double pole, and 
$\left( z-w_{j}\right) ^{2}{g(z) \rightarrow -}\frac{1}{4}$ as $z\rightarrow w_{j}$.

We call these \textit{\ admissible poles}. For $j=0,\pm 1$ we then choose an arbitrary $z^{(j) }\in T_{j}\cap Z$. These will be either at an admissible pole, or at infinity if $f(z)$ and $g(z)$ can be expanded in convergent series in a neighborhood of $z=\infty $ of the form
\begin{equation}
f(z)=z^{m}\sum\limits_{s=0}^{\infty }f_{s}z^{-s},\ g(z)=z^{p}\sum\limits_{s=0}^{\infty }g_{s}{z}^{-s},
\label{fginfinity}
\end{equation}
where $f_{0}\neq 0$, $g_{0}\neq 0$, and either $m$ and $p$ are integers such that $m>-2$ and $p<\frac{1}{2}m-1$, or $m=p=-2$ and $g_{0}=-\frac{1}{4}$. For details and generalizations of (\ref{fginfinity}) see \cite[Chap. 10, Sects. 4 and 5]{Olver:1997:ASF}.

For each\ $j=0,\pm 1$\ the following LG region of validity $Z_{j}(u,z^{(j)})$ (abbreviated $Z_{j}$) then comprises the $z$ point set for which there is a \textit{progressive} path $\hat{\mathcal{L}}_{j}$ linking $z$ with $z^{(j) }$ in $Z$ and having the properties (i) $\hat{\mathcal{L}}_{j} $ consists of a finite chain of $R_{2}$ arcs (as defined in \cite[Chap. 5, sec. 3.3]{Olver:1997:ASF}), and (ii) as $v$ passes along $\hat{\mathcal{L}}_{j} $ from $z^{(j) }$ to $z$, the real part of $(-1)^{j}u\xi(v)$ is nonincreasing,
where $\xi(v)$ is given by (\ref{xi}) with $z=v$, and with the chosen sign fixed throughout.

\begin{figure}[htbp]
  \centering
  \includegraphics{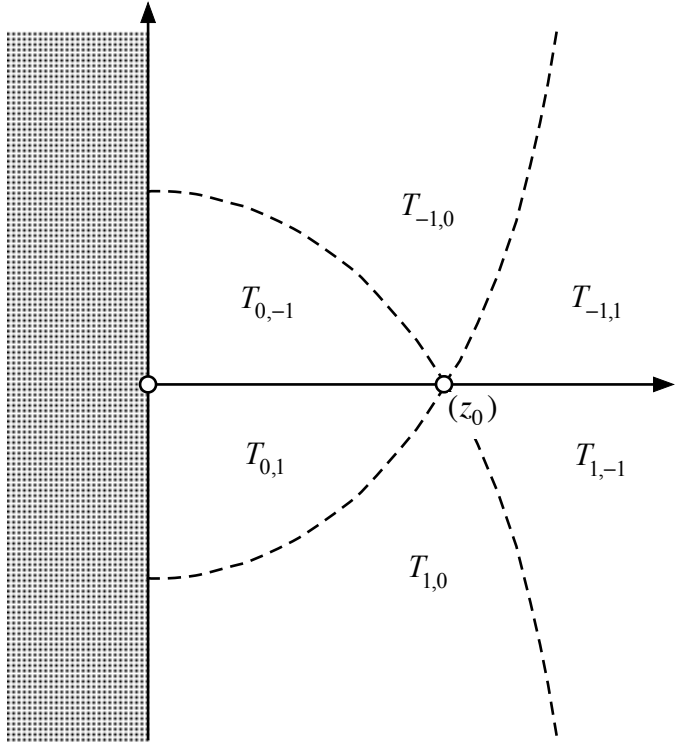}
  \caption{Regions $T_{j,k}$ in $z$ plane for ${u}$ positive.}
  \label{fig:fig1}
\end{figure}

We assume throughout this paper the following.

\begin{hypothesis}
\label{hyp:1} 
For each $z^{(j)}\in T_{j}\cap Z_{j}$ ( $j=0,\pm 1$) assume $z^{\left( 0\right) }\in Z_{1}\cap Z_{-1}$ and $z^{\left( \pm 1\right) }\in Z_{0}\cap Z_{\mp 1}$, i.e. for $j,k=0,\pm 1$ there is a path, consisting of a finite chain of $R_{2}$ arcs, linking $z^{\left( j\right) }$ with $z^{\left( k\right) }$ in $Z$ such as $z$ passes along the path from $z^{\left( j\right) }$ to $z^{\left( k\right) }$, the real part of $u\xi $ is monotonic (with $\xi$ varying continuously). 
\end{hypothesis} 

We now define certain terms which will appear in our expansions. Firstly
\begin{equation} 
\Phi (z) =\frac{4f(z) {f}^{\prime \prime }(z) -5{f}^{\prime 2}(z) }{16f^{3}(z) }+\frac{ g(z) }{f(z) }, 
\label{eq5} 
\end{equation}
and then from \cite{Dunster:2017:COA} the set of coefficients 
\begin{equation} 
\hat{F}_{1}(z) ={\tfrac{1}{2}}\Phi (z) ,\ \hat{F}_{2}(z) =-{\tfrac{1}{4}}f^{-1/2}(z) {\Phi }^{\prime }(z) , 
\label{eq8} 
\end{equation}
and 
\begin{equation} 
\hat{F}_{s+1}(z) =-\tfrac{1}{2}f^{-1/2}(z) \hat{F}_{s}^{\prime }(z) -\tfrac{1}{2}\sum\limits_{j=1}^{s-1}{\hat{F}_{j}(z) \hat{F}_{s-j}(z) }\ \left( {s=2,3,4\cdots } \right) . 
\label{eq9} 
\end{equation}

Next define 
\begin{equation} 
\hat{E}_{2s+1}(z) =\int {\hat{F}_{2s+1}(z) f^{1/2}(z) dz}\ \left( {s=0,1,2,\cdots }\right) , 
\label{eq7} 
\end{equation}
where the integration constants must be chosen so that each $\left( z-z_{0}\right)^{1/2}\hat{E}_{2s+1}(z) $ is meromorphic (non-logarithmic) at the turning point. 

Although (\ref{eq7}) holds for $2s+1$ replaced by $2s+2$, the even ones do not require an integration. Instead they can be determined via the formal expansion 
\begin{equation} 
\sum\limits_{s=1}^{\infty }{\dfrac{\hat{E}_{2s}(z) }{u^{2s}} \sim -}\frac{1}{2}\ln \left\{ 1+\sum\limits_{s=0}^{\infty }{\dfrac{\hat{F}_{2s+1}(z) }{u^{2s+2}}}\right\} +\sum\limits_{s=1}^{\infty }{ \dfrac{{\alpha}_{2s}}{u^{2s}}}, 
\label{even} 
\end{equation}
where each ${\alpha }_{2s}$ can be arbitrarily chosen.

We next define two sequences $\left\{ a_{s}\right\}_{s=1}^{\infty }$ and $ \left\{ \tilde{a}_{s}\right\}_{s=1}^{\infty }$ by $a_{1}=a_{2}=\frac{5}{72}$, $\tilde{a}_{1}=\tilde{a}_{2}=-{\frac{7}{72}}$, with subsequent terms $a_{s}$ and $\tilde{a}_{s}$ ($s=2,3,\cdots $) satisfying the same recursion formula 
\begin{equation} 
a_{s+1}=\tfrac{1}{2}\left( {s+1}\right) a_{s}+\tfrac{1}{2} \sum\limits_{j=1}^{s-1}{a_{j}a_{s-j}}. 
\label{arec} 
\end{equation}

Then let 
\begin{equation} 
\mathcal{E}_{s}(z) =\hat{E}_{s}(z)+(-1)^{s}a_{s}s^{-1}\xi ^{-s}, 
\label{eq40} 
\end{equation}
and 
\begin{equation} 
\tilde{\mathcal{E}}_{s}(z) =\hat{E}_{s}(z) +(-1)^{s}\tilde{a}_{s}s^{-1}\xi ^{-s}. 
\label{eq38} 
\end{equation}

The coefficients $a_{s}$ and $\tilde{a}_{s}$ are the ones that appear in the exponential form of the LG expansions for the Airy function and its derivative \cite[Appendix A]{Dunster:2020:SEB}. The following theorem demonstrates that these LG expansions are Borel summable, that is, the coefficients are $\mathcal{O}(K^s s!)$ as $s \rightarrow \infty$ for some $K>0$. The proof is given in \cref{secA}.

\begin{theorem}
\label{thm:as}
$a_{s}>0$ and $\tilde{a}_{s}<0$ for all $s$, and moreover 
\begin{equation} 
\tfrac{5}{36}\left( \tfrac{1}{2}\right) ^{s}s!\leq a_{s}\leq \tfrac{5}{36}\left( \tfrac{4453}{6912}\right) ^{s}s!, 
\label{eq19} 
\end{equation} 
and 
\begin{equation} 
\tfrac{7}{36}\left( \tfrac{1}{2}\right) ^{s}(s-1)!\leq \left\vert \tilde{a}_{s}\right\vert \leq \tfrac{7}{36}\left( \tfrac{1}{2}\right) ^{s}s!. 
\label{eq19a} 
\end{equation} 
\end{theorem}

The plan of the paper is as follows. In \cref{sec2} we obtain sharp error bounds for $z$ not too close to the turning point. In \cref{sec3} we use a similar method, along with Cauchy's integral formula, to obtain comparable bounds in a domain containing the turning point. This method involves majorizing the coefficients over a Cauchy integral contour. We overcome a problem of the coefficients having a large amplitude and oscillation over these contours by employing a modification of Cauchy's integral formula for meromorphic functions. Finally in \cref{sec4} we illustrate numerically the sharpness of our new bounds by applying them to solutions of Bessel's equation.

\section{Error bounds away from the turning point} 
\label{sec2} 

We begin by defining terms that appear in the error bounds. Let 
\begin{multline}
\omega_{n,j}(u,z) =2\int_{z^{(j)}}^{z}{\left\vert {
\hat{F}_{n}(t) f^{1/2}(t) dt}\right\vert } \\ 
+\sum\limits_{s=1}^{n-1}\dfrac{1}{\left\vert u\right\vert ^{s}}{
\int_{z^{(j)}}^{z}{\left\vert {\sum\limits_{k=s}^{n-1}{\hat{F}
_{k}(t) \hat{F}_{s+n-k-1}(t) }f^{1/2}(t) dt}\right\vert }},
\label{eq13}
\end{multline}

\begin{equation}
\varpi_{n,j}(u,z) =4\sum\limits_{s=0}^{n-2}\frac{1}{{
\left\vert u\right\vert ^{s}}}{\int_{z^{(j)}}^{z}{\left\vert {\hat{F}_{s+1}(t) f^{1/2}(t) dt}\right\vert }},
\label{eq14}
\end{equation}

\begin{equation} 
\gamma_{n}(u,\xi) =\frac{2a_{n}\Lambda_{n+1}}{\left\vert\xi \right\vert ^{n}}+\frac{1}{\left\vert u\right\vert \left\vert \xi\right\vert ^{n+1}}\sum\limits_{s=0}^{n-2}{\frac{\Lambda_{n+s+2}}{\left\vert {u\xi }\right\vert ^{s}}\sum\limits_{k=s+1}^{n-1}{a_{k}a_{s+n-k}}}, 
\label{eq28} 
\end{equation}

\begin{equation} 
{\beta }_{n}(u,\xi) =\frac{4}{\left\vert \xi \right\vert }\sum\limits_{s=0}^{n-2}{\frac{a_{s+1}\Lambda_{s+2}}{\left\vert {u\xi }\right\vert ^{s}}}, 
\label{eq29} 
\end{equation}

\begin{equation} 
\tilde{\gamma}_{n}(u,\xi) =\frac{2\left\vert \tilde{a}_{n}\right\vert \Lambda_{n+1}}{\left\vert \xi \right\vert ^{n}}+\frac{1}{\left\vert u\right\vert \left\vert \xi \right\vert ^{n+1}}\sum\limits_{s=0}^{n-2}{\frac{\Lambda_{n+s+2}}{\left\vert {u\xi }\right\vert ^{s}}\sum\limits_{k=s+1}^{n-1}\tilde{a}{_{k}\tilde{a}_{s+n-k}}}, 
\label{eq28a} 
\end{equation}

\begin{equation} 
{\tilde{\beta}}_{n}(u,\xi) =\frac{4}{\left\vert \xi\right\vert }\sum\limits_{s=0}^{n-2}{\frac{\left\vert \tilde{a}_{s+1}\right\vert \Lambda_{s+2}}{\left\vert {u\xi }\right\vert ^{s}}}, 
\label{eq29a} 
\end{equation}
where 
\begin{equation}
\Lambda_{n}=\dfrac{\pi ^{1/2}\Gamma \left( \frac{1}{2}n-\frac{1}{2}\right) }{2\Gamma \left( \frac{1}{2}n\right) }. 
\label{eq93} 
\end{equation}
The paths of integration in (\ref{eq13}) and (\ref{eq14}) are taken along $hat{\mathcal{L}}_{j}$.

We now define three solutions $w_{j}(u,z) $ ($j=0,\pm 1$) of (\ref{eq1}) having the unique properties 
\begin{equation} 
w_{j}(u,z) \sim f^{-1/4}(z) e^{-u \xi}\;(z\rightarrow z^{(j)}). 
\label{eq2a} 
\end{equation}

We can assume there exist connection coefficients $\lambda_{\pm 1}$ such that the following relation holds 
\begin{equation}
\lambda_{-1}w_{-1}(u,z) =iw_{0}(u,z) +\lambda_{1}w_{1}(u,z).
\label{eq15} 
\end{equation}
With $\lambda_{\pm 1}$ specified we then define the two constants 
\begin{equation}
\delta_{n,\pm 1}(u) =\lambda_{\pm 1}\exp \left\{ \sum\limits_{s=1}^{n-1}{\frac{{\left( {-1}\right) ^{s}}\hat{E}_{s}\left( {z^{\left( 0\right) }}\right) -\hat{E}_{s}\left( {z^{\left( \pm 1\right) }}\right) }{u^{s}}}\right\} -1, 
\label{exactdelta} 
\end{equation}
and as shown in \cite{Dunster:2020:SEB} these are $\mathcal{O}\left( u^{-n}\right) $ as $u\rightarrow \infty $ under \cref{hyp:1}.

If the connection coefficients $\lambda_{\pm 1}$ of (\ref{eq15}) are not known explicitly, in the error bounds that follow we instead can replace $|\delta_{n,\pm 1}(u)|$ with the bound 
\begin{equation}
\label{Om}
\left\vert \delta_{n,\pm 1}(u) \right\vert \leq \dfrac{2\Omega_{n}(u)}{\left\vert u\right\vert ^{n}-\Omega_{n}(u)}, 
\end{equation}
where 
\begin{equation}
{\Omega_{n}(u) =}\max_{j,k}\left[ \omega_{n,j}\left( u,z^{(k)}\right) \exp \left\{ {\left\vert u\right\vert^{-1}\varpi_{n,j}\left( u,z^{(k)}\right) +\left\vert u\right\vert ^{-n}\omega_{n,j}\left( u,z^{(k)}\right) }\right\} \right] . 
\end{equation}
Here the maximum is taken over $j,k\in \{0,\pm 1 \}$ such that $j\neq k$. In this case $|u|$ must be assumed to be sufficiently large so that the denominator of the RHS of (\ref{Om}) is positive; note that $\Omega_{n}(u)=\mathcal{O}(1)$ as $u \rightarrow \infty$.

Next define 
\begin{multline}
 d_{n}(u,z) =\exp \left\{ \sum\limits_{s=1}^{n-1}\mathrm{Re}{\dfrac{\mathcal{E}_{s}(z) }{u^{s}}}\right\} e_{n,j}(u,z) \left\{ 1+\dfrac{e_{n,j}(u,z) }{2{\left\vert u\right\vert ^n}}\right\} ^{2} \\+\exp \left\{ \sum\limits_{s=1}^{n-1}{\left( {-1}\right) ^{s}\mathrm{Re}\dfrac{\mathcal{E}_{s}(z) }{u^{s}}}\right\} e_{n,k}(u,z) \left\{ 1+\dfrac{e_{n,k}(u,z) }{2{\left\vert u\right\vert ^{n}}}\right\} ^2, 
\label{dB} 
\end{multline}
and
\begin{multline}
 \tilde{d}_{n}(u,z) =\exp \left\{ \sum\limits_{s=1}^{n-1}\mathrm{Re}{\dfrac{\mathcal{\tilde{E}}_{s}(z) }{u^{s}}}\right\} \tilde{e}_{n,j}(u,z) \left\{ 1+\dfrac{\tilde{e}_{n,j}(u,z) }{2{\left\vert u\right\vert ^{n}}}\right\} ^{2} \\+\exp \left\{ \sum\limits_{s=1}^{n-1}\left( {-1}\right) ^{s}{\mathrm{Re}\dfrac{\mathcal{\tilde{E}}_{s}(z) }{u^{s}}}\right\} \tilde{e}_{n,k}(u,z) \left\{ 1+\dfrac{\tilde{e}_{n,k}(u,z) }{2{\ \left\vert u\right\vert ^{n}}}\right\} ^{2}, 
\label{dA} 
\end{multline}
 where
 \begin{multline}
 e_{n,j}(u,z) ={\left\vert u\right\vert ^{n}}\left\vert \delta_{n,j}(u) \right\vert  +\omega_{n,j}(u,z) \exp \left\{ {\left\vert u\right\vert ^{-1}\varpi_{n,j}(u,z) +\left\vert u\right\vert ^{-n}\omega_{n,j}(u,z) }\right\} \\ +\gamma_{n}(u,\xi) \exp \left\{ {\left\vert u\right\vert^{-1}\beta_{n}(u,\xi) +\left\vert u\right\vert ^{-n}\gamma_{n}(u,\xi) }\right\}, 
\label{eq100} 
\end{multline}
and
\begin{multline}
 \tilde{e}_{n,j}(u,z) ={\left\vert u\right\vert ^{n}}\left\vert \delta_{n,j}(u) \right\vert +\omega_{n,j}(u,z) \exp \left\{ {\left\vert u\right\vert ^{-1}\varpi_{n,j}(u,z) +\left\vert u\right\vert ^{-n}\omega_{n,j}(u,z) }\right\} \\ +\tilde{\gamma}_{n}(u,\xi) \exp \left\{ {\left\vert u\right\vert ^{-1}\tilde{\beta}_{n}(u,\xi) +\left\vert u\right\vert ^{-n}\tilde{\gamma}_{n}(u,\xi) }\right\}.
\label{eq103} 
\end{multline}
 In (\ref{dA}) and (\ref{dB}) $j=\pm 1$, $k=0$ for $z\in T_{0,\pm 1}\cup T_{\pm 1,0}$, and $j=\pm 1$, $k=\mp 1$ for $z\in T_{\pm 1,\mp 1}$. 
 
 Our main result reads as follows.
 \begin{theorem}
 \label{thm:far}
 Assume \cref{hyp:1}, and let $z\in Z_{j}\cap Z_{k}$ ($j,k\in \left\{ 0,1,-1\right\} $, $j\neq k$). Then for positive integers $m$ and $r$ the differential equation (\ref{eq1}) has solutions 
\begin{equation}
w_{m,l}(u,z) =\mathrm{Ai}_{l}\left( u^{2/3}\zeta\right) \mathcal{A}_{2m+2}(u,z) +\mathrm{Ai}_{l}^{\prime }\left( u^{2/3}\zeta\right) \mathcal{B}_{2m+1}(u,z) \ (l=0,\pm 1), 
\label{wjs} 
\end{equation}
where 
\begin{multline} \left\{ \dfrac{f(z) }{\zeta }\right\} ^{1/4}\mathcal{A}_{2m+2}(u,z) ={\exp \left\{ \sum\limits_{s=1}^{m}{\dfrac{\mathcal{\tilde{E}}_{2s}(z) }{u^{2s}}}\right\} \cosh \left\{ \sum\limits_{s=0}^{m}{\dfrac{\mathcal{\tilde{E}}_{2s+1}(z) }{u^{2s+1}}}\right\} }\\+\dfrac{1}{2}\tilde{\varepsilon}_{2m+2,r}(u,z), 
\label{1.1} 
\end{multline}

\begin{multline} u^{1/3}\left\{ \zeta f(z) \right\} ^{1/4}\mathcal{B}_{2m+1}(u,z) ={\exp \left\{ \sum\limits_{s=1}^{m}{\dfrac{\mathcal{E}_{2s}(z) }{u^{2s}}}\right\} \sinh \left\{ \sum\limits_{s=0}^{m-1}{\dfrac{\mathcal{E}_{2s+1}(z) }{u^{2s+1}}}\right\} }\\+\dfrac{1}{2}\varepsilon_{2m+1,r}(u,z), 
\label{1.2} 
\end{multline}
in which for $z \neq z_{0}$
\begin{multline}
 \left\vert \tilde{\varepsilon}_{2m+2,r}(u,z) \right\vert \leq {\exp \left\{ \sum\limits_{s=1}^{2m+1}\mathrm{Re}{\dfrac{\mathcal{\tilde{E}}_{s}(z) }{u^{s}}}\right\} }\left\vert \exp \left\{ \sum\limits_{s=2m+2}^{2m+2r+1}{\dfrac{\mathcal{\tilde{E}}_{s}(z) }{u^{s}}}\right\} {-1}\right\vert \\ +\exp \left\{ \sum\limits_{s=1}^{2m+1}(-1) ^{s}{\mathrm{Re}\dfrac{\mathcal{\tilde{E}}_{s}(z) }{u^{s}}}\right\} \left\vert\exp \left\{ \sum\limits_{s=2m+2}^{2m+2r+1}(-1) ^{s}{\dfrac{\mathcal{\tilde{E}}_{s}(z) }{u^{s}}}\right\} {-1}\right\vert +\frac{\tilde{d}_{2m+2r+2}(u,z) }{\left\vert u\right\vert^{2m+2r+2}}, 
\label{1.3} 
\end{multline}
 and 
\begin{multline}
 \left\vert \varepsilon_{2m+1,r}(u,z) \right\vert \leq {\exp\left\{ \sum\limits_{s=1}^{2m}\mathrm{Re}{\dfrac{\mathcal{E}_{s}\left(z\right) }{u^{s}}}\right\} }\left\vert \exp \left\{ \sum\limits_{s=2m+1}^{2m+2r+1}{\dfrac{\mathcal{E}_{s}(z) }{u^{s}}}\right\} {-1}\right\vert \\ +{\exp \left\{ {\sum\limits_{s=1}^{2m}}(-1) ^{s}\mathrm{Re}{\dfrac{\mathcal{E}_{s}(z) }{u^{s}}}\right\} }\left\vert {\exp\left\{ \sum\limits_{s=2m+1}^{2m+2r+1}(-1) ^{s}{\dfrac{\mathcal{E}_{s}(z) }{u^{s}}}\right\} {-1}}\right\vert +\frac{d_{2m+2r+2}(u,z) }{\left\vert u\right\vert ^{2m+2r+2}}. 
\label{1.4} 
\end{multline}

\end{theorem} 

\begin{proof}
From \cite[Thm. 3.4]{Dunster:2020:SEB} replacing $m$ by $m+r$ we have the following three solutions of (\ref{eq1}) 
\begin{equation}
\mathrm{\ Ai}_{l}\left( u^{2/3}\zeta\right) \mathcal{A}_{2m+2r+2}(u,z) +\mathrm{Ai}_{l}^{\prime }\left( u^{2/3}\zeta\right) \mathcal{B}_{2m+2r+2}(u,z) \ (l=0,\pm 1), 
\label{1.5} 
\end{equation}
where 
\begin{multline}
 \left\{ \dfrac{f(z) }{\zeta }\right\} ^{1/4}\mathcal{A}_{2m+2r+2}(u,z) ={\exp \left\{ \sum\limits_{s=1}^{m+r}{\dfrac{\mathcal{\tilde{E}}_{2s}(z) }{u^{2s}}}\right\} \cosh \left\{ \sum\limits_{s=0}^{m+r}{\dfrac{\mathcal{\tilde{E}}_{2s+1}(z) }{u^{2s+1}}}\right\} } \\ +\dfrac{1}{2}\tilde{\varepsilon}_{2m+2r+2}(u,z) , 
\label{oldA} 
\end{multline}
 and 
\begin{multline}
 u^{1/3}\left\{ \zeta f(z) \right\} ^{1/4}\mathcal{B}_{2m+2r+2}(u,z) ={\exp \left\{ \sum\limits_{s=1}^{m+r}{\dfrac{\mathcal{E}_{2s}(z) }{u^{2s}}}\right\} \sinh \left\{ \sum\limits_{s=0}^{m+r}{\dfrac{\mathcal{E}_{2s+1}(z) }{u^{2s+1}}}\right\} } \\ +\dfrac{1}{2}\varepsilon_{2m+2r+2}(u,z) . 
\label{oldB} 
\end{multline}

 The error terms $\tilde{\varepsilon}_{2m+2r+2}(u,z) $ and $\varepsilon_{2m+2r+2}(u,z) $ are bounded by \cite[Eqs. (3.29) - (3.32)]{Dunster:2020:SEB} and in particular 
\begin{equation}
\left\vert \tilde{\varepsilon}_{2m+2r+2}(u,z) \right\vert \leq \frac{\tilde{d}_{2m+2r+2}(u,z) }{\left\vert u\right\vert^{2m+2r+2}}, 
\label{1.5a} 
\end{equation}
and 
\begin{equation}
\left\vert \varepsilon_{2m+2r+2}(u,z) \right\vert \leq \frac{d_{2m+2r+2}(u,z) }{\left\vert u\right\vert ^{2m+2r+2}}. 
\label{1.5b} 
\end{equation}
Then we relabel $\mathcal{A}_{2m+2r+2}(u,z) $ by $\mathcal{A}_{2m+2}(u,z) $, and then recast it in the form 
\begin{multline}
 \left\{ \dfrac{f(z) }{\zeta }\right\} ^{1/4}\mathcal{A}_{2m+2}(u,z) ={\exp \left\{ \sum\limits_{s=1}^{m}{\dfrac{ \mathcal{\tilde{E}}_{2s}(z) }{u^{2s}}}\right\} \cosh \left\{ \sum\limits_{s=0}^{m}{\dfrac{\mathcal{\tilde{E}}_{2s+1}(z) }{ u^{2s+1}}}\right\} } \\ +\dfrac{1}{2}\tilde{\varepsilon}_{2m+2,r}(u,z) , 
\label{1.6} 
\end{multline}
 Comparing this to (\ref{oldA}) implies that 
\begin{multline}
 \tilde{\varepsilon}_{2m+2,r}(u,z) =2{\exp \left\{  \sum\limits_{s=1}^{m+r}{\dfrac{\mathcal{\tilde{E}}_{2s}(z) }{ u^{2s}}}\right\} \cosh \left\{ \sum\limits_{s=0}^{m+r}{\dfrac{\mathcal{ \tilde{E}}_{2s+1}(z) }{u^{2s+1}}}\right\} } \\ -2{\exp \left\{\sum\limits_{s=1}^{m}{\dfrac{\mathcal{\tilde{E}}_{2s}(z) }{u^{2s}}}\right\} \cosh \left\{ \sum\limits_{s=0}^{m}{\dfrac{ \mathcal{\tilde{E}}_{2s+1}(z) }{u^{2s+1}}}\right\} }+\tilde{ \varepsilon}_{2m+2r+2}(u,z) . 
\label{1.7} 
\end{multline}

 Similarly on relabeling $\mathcal{B}_{2m+2r+2}(u,z)$ by $\mathcal{B}_{2m+1}(u,z)$ and following the same procedure yields 
\begin{multline}
 u^{1/3}\left\{ \zeta f(z) \right\} ^{1/4}\mathcal{B}_{2m+1}(u,z) ={\exp \left\{ \sum\limits_{s=1}^{m}{\dfrac{ \mathcal{E}_{2s}(z) }{u^{2s}}}\right\} \sinh \left\{ { \sum\limits_{s=0}^{m-1}{\dfrac{\mathcal{E}_{2s+1}(z) }{u^{2s+1}}} }\right\} } \\ +\dfrac{1}{2}\varepsilon_{2m+1,r}(u,z), 
\label{1.8} 
\end{multline}
 where 
\begin{multline}
 \varepsilon_{2m+1,r}(u,z) =2{\exp \left\{\sum \limits_{s=1}^{m+r}{\dfrac{\mathcal{E}_{2s}(z) }{u^{2s}}} \right\} \sinh \left\{ \sum\limits_{s=0}^{m+r}{\dfrac{\mathcal{E}_{2s+1}(z) }{u^{2s+1}}}\right\} } \\ -2{\exp \left\{ \sum\limits_{s=1}^{m}{\dfrac{\mathcal{E}_{2s}(z) }{u^{2s}}}\right\} \sinh \left\{ {\sum\limits_{s=0}^{m-1}\dfrac{ \mathcal{E}_{2s+1}(z) }{u^{2s+1}}}\right\} }+\varepsilon_{2m+2r+2}(u,z) . 
\label{1.9} 
\end{multline}

 Now we write (\ref{1.7}) in the form 
\begin{multline}
 \tilde{\varepsilon}_{2m+2,r}(u,z) ={\exp \left\{ \sum\limits_{s=1}^{2m+1}{\dfrac{\mathcal{\tilde{E}}_{s}(z) }{ u^{s}}}\right\} }\left[ \exp \left\{ \sum\limits_{s=2m+2}^{2m+2r+1}{\dfrac{ \mathcal{\tilde{E}}_{s}(z) }{u^{s}}}\right\} {-1}\right] \\ +{\exp \left\{ {\sum\limits_{s=1}^{2m+1}}(-1) ^{s}{\dfrac{ \mathcal{\tilde{E}}_{s}(z) }{u^{s}}}\right\} }\left[ \exp \left\{ \sum\limits_{s=2m+2}^{2m+2r+1}(-1) ^{s}{\dfrac{\mathcal{ \tilde{E}}_{s}(z) }{u^{s}}}\right\} {-1}\right] +\tilde{ \varepsilon}_{2m+2r+2}(u,z) . 
\label{1.10} 
\end{multline}
On taking absolute values of both sides, and using the triangle inequality along with (\ref{1.5a}) yields (\ref{1.3}). The bound (\ref{1.4}) follows similarly from (\ref{1.5b}) and (\ref{1.8}). 
\end{proof}

\begin{remark}
\label{remark2}
For each of the terms in the new bounds that involve the difference between $1$ and an exponential having a small argument we can numerically make use of the identity 
\begin{equation} e^{w}-1=w+\tfrac{1}{2}w^{2}+w^{3}g(w),
\label{e1} 
\end{equation}
where 
\begin{equation}
g(w) =\sum_{n=0}^{\infty }\frac{w^{n}}{\left( n+3\right) !}= \frac{2e^{w}-2-2w-w^{2}}{2w^{3}}.
\label{e2} 
\end{equation}
From its Maclaurin series we observe that $g(w) $ is monotonically increasing for positive $w$, and moreover 
\begin{equation}
\left\vert g(w) \right\vert \leq \sum_{n=0}^{\infty }\frac{ \left\vert w\right\vert ^{n}}{\left( n+3\right) !}=g\left( \left\vert w\right\vert \right).
\label{e3} 
\end{equation}
Hence for $\left\vert w\right\vert \leq w_{0}$ we have $\left\vert g(w) \right\vert \leq g\left( \left\vert w\right\vert \right) \leq g\left( w_{0}\right) $ and therefore 
\begin{equation}
\left\vert e^{w}-1\right\vert \leq \left\vert w\right\vert +\tfrac{1}{2} \left\vert w\right\vert ^{2}+\left\vert w\right\vert ^{3}\left\vert g(w) \right\vert \leq \left\vert w\right\vert +\tfrac{1}{2}\left\vert w\right\vert ^{2}+\left\vert w\right\vert ^{3}g\left( w_{0}\right).
\label{e4} 
\end{equation}
For example if $\left\vert w\right\vert \leq 1$, we use $g(1) =0.218\cdots $, and so we have 
\begin{equation}
\left\vert e^{w}-1\right\vert \leq \left\vert w\right\vert +\tfrac{1}{2} \left\vert w\right\vert ^{2}+\left( 0.218\cdots \right) \left\vert w\right\vert ^{3}.
\label{e5} 
\end{equation}

\end{remark}
 
\begin{remark}
From (\ref{1.3}) and (\ref{1.4}) we observe that 
\begin{equation}
\tilde{\varepsilon}_{2m+2,r}(u,z) =\mathcal{O}\left( u^{-2m-2}\right),
\label{1.11} 
\end{equation}
and 
\begin{equation}
\varepsilon_{2m+1,r}(u,z) =\mathcal{O}\left( u^{-2m-1}\right),
\label{1.12} 
\end{equation}
as $u\rightarrow \infty $ for $z\in Z_{j}\cap Z_{k}$, where the pair $(j,k)$ can take any value from $(0,1)$, $(-1,0)$ and $(-1,1)$. Note the bounds do not exist at the turning point, since $\Phi(z)$ has a pole at that point (see (\ref{eq5})), and hence the integrals in (\ref{eq13}) and (\ref{eq13}) diverge there. Thus the bounds hold uniformly in $Z_{j}\cap Z_{k}$ with a neighborhood of $z=z_{0}$ removed. In the next section we obtain similar sharp bounds that are valid at the turning point.

\end{remark} 

\begin{remark}
The choice of $r$ controls the accuracy of the bounds, with a larger value generally giving sharper values relative to the exact error. However if it is taken too large the accuracy may deteriorate, since the series is asymptotic with finite accuracy possible.

In the unmodified case of \cite{Dunster:2020:SEB} ($r=0$) the bound on $\varepsilon_{2m+2}(u,z)$, while of the correct order of magnitude, can significantly overestimate the exact error. This is primarily because it involves the suprema of functions having high amplitudes and oscillations. This is typical for these type of error bounds. For $\tilde{\varepsilon}_{2m+2,r}(u,z)$ this is exacerbated by the fact that the original bound in of \cite{Dunster:2020:SEB} overestimates the exact error by a factor $\mathcal{O}(u)$. Of course these and other error bounds are still important as they rigorously establish the veracity of the asymptotic approximations in question. See \cite{Olver:1980:AAA} for a discussion on the importance of explicit error bounds, even if they are hard to compute, or are computable but not sharp.

While the original bounds of \cite{Dunster:2020:SEB} do still appear in the new bounds (the $d$ and $\tilde{d}$ terms in (\ref{1.3}) and (\ref{1.4})), they have been "tamed" by a reduction of a factor $\mathcal{O}(u^{-2r})$. In addition, as we shall demonstrate in \cref{sec4}, the terms preceding them in these bounds are close in absolute value to the exact error.
\end{remark}

\section{Error bounds close to the turning point} 
\label{sec3} 

Let $\Gamma $ be the circle $\{z:\left\vert z-z_{0}\right\vert =r_{0}\} $ for $r_{0}>0$ is arbitrarily chosen but not too small, and such that the loop lies in the intersection of $Z_{0}$, $Z_{1}$, and $Z_{-1}$. In addition, let $\gamma_{j,l}$ be the union of part of the loop $\Gamma $ that lies in $T_{j,l}$ ($j,l\in \left\{ 0,1,-1\right\} ,j\neq l$) with an arbitrarily chosen progressive path in $T_{j}$ connecting $\Gamma $ to $ z^{\left( j\right) }$ (if possible a straight line). There are six such paths, examples of two of which are shown in \Cref{fig:fig2,fig:fig3}. In these figures $\mathrm{Re}\, z \geq 0$, $u >0$, $z_{0}>0$, $z^{(0)}$ is an admissible pole at the origin, and $z^{(1)}$ is at infinity.

\begin{figure}
  \centering
  \includegraphics{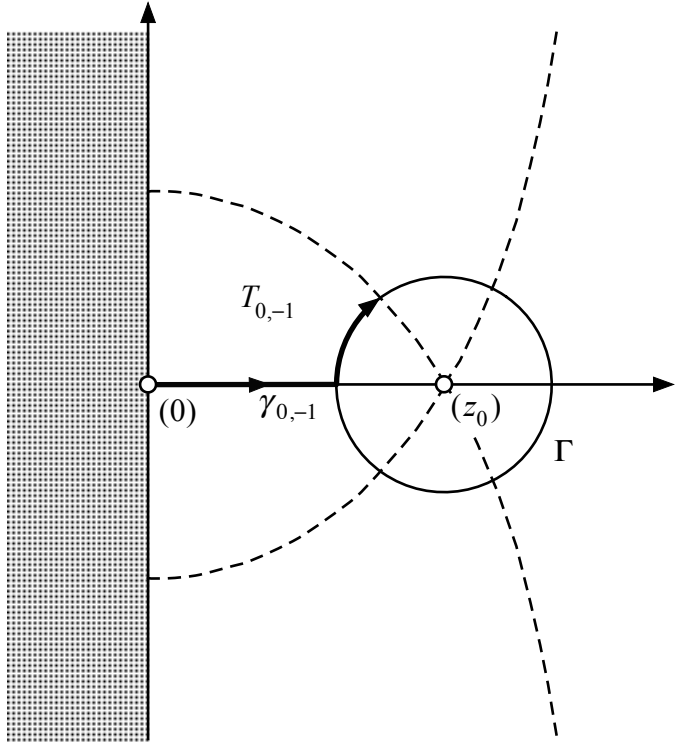}
  \caption{Path $\gamma_{0,-1}$ in the $z$ plane.}
  \label{fig:fig2}
\end{figure}

\begin{figure}
  \centering
  \includegraphics{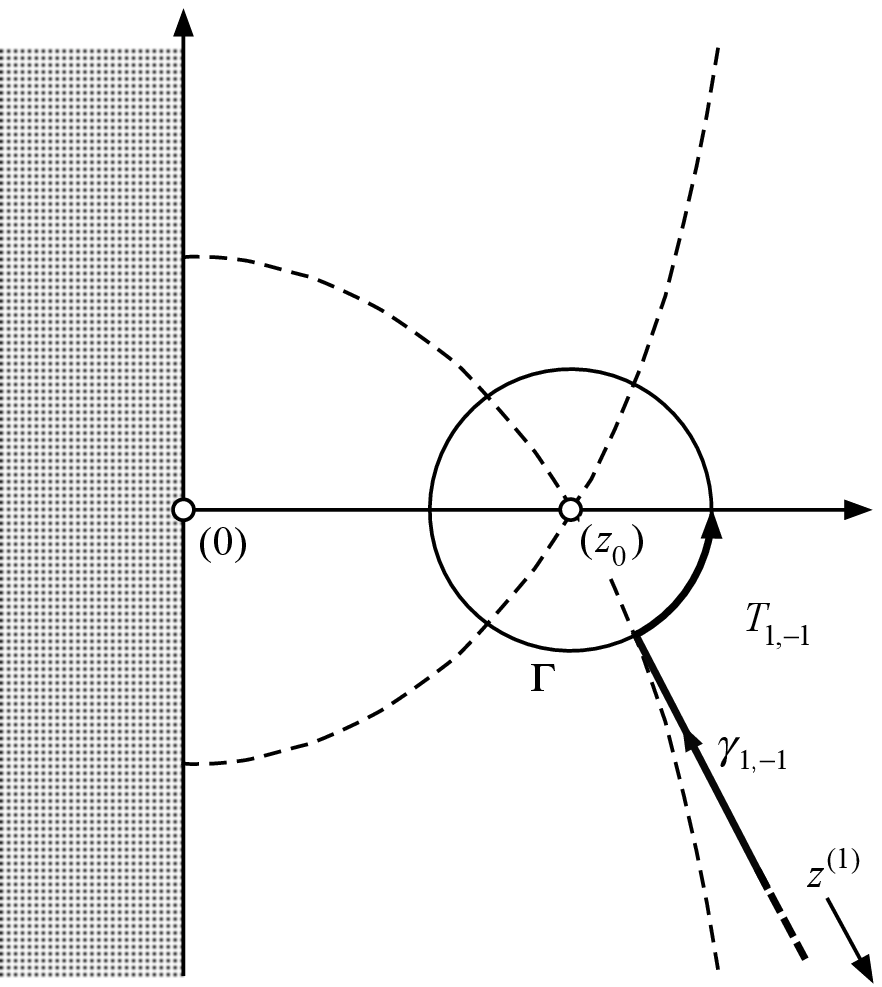}
  \caption{Path $\gamma_{1,-1}$ in the $z$ plane.}
  \label{fig:fig3}
\end{figure}

We now define 
\begin{multline}
 \omega_{n}(u) =2\max_{j,l}\left\{ \int_{\gamma_{j,l}}{ \left\vert {\hat{F}_{n}(t) f^{1/2}(t) dt} \right\vert }\right\} \\ +\sum\limits_{s=1}^{n-1}\dfrac{1}{\left\vert u\right\vert ^{s}}{\ \sum\limits_{k=s}^{n-1}\max_{j,l}}\left\{ {\int_{\gamma_{j,l}}{\left\vert {{ \hat{F}_{k}(t) \hat{F}_{s+n-k-1}(t) } f^{1/2}(t) dt}\right\vert }}\right\} , 
\label{eq68} 
\end{multline}
 and 
\begin{equation}
\varpi_{n}(u) =4\sum\limits_{s=0}^{n-2}\frac{1}{\left\vert u\right\vert ^{s}}{\max_{j,l}}\left\{ {\int_{\gamma_{j,l}}{\left\vert { \hat{F}_{s+1}(t) f^{1/2}(t) dt}\right\vert }} \right\} , 
\label{eq69} 
\end{equation}
where the maxima are taken over all six paths $\gamma_{j,l}$

Next 
\begin{equation}
\delta_{n}(u) =\max_{j=\pm 1}\left\vert \delta_{n,j}(u) \right\vert , 
\label{deltan} 
\end{equation}

\begin{equation}
\Upsilon =\underset{z\in \Gamma }{\inf }\left\vert \zeta f(z) \right\vert ^{1/4},\ \tilde{\Upsilon}=\underset{z\in \Gamma }{\sup } \left\vert \zeta /f(z) \right\vert ^{1/4}, 
\label{zeds} 
\end{equation}
and 
\begin{equation}
\rho =\underset{z\in \Gamma }{\inf }\left\vert \xi \right\vert . 
\label{rho} 
\end{equation}

Let $\theta =\arg u$, and we further define 
\begin{equation}
M_{s}=\underset{z\in \Gamma }{\sup }\,\mathrm{Re}\left\{ e^{-is\theta } \mathcal{E}_{s}(z) \right\} ,\ N_{s}=\underset{z\in \Gamma }{ \sup}\,\mathrm{Re}\left\{ {\left( {-1}\right) ^{s}}e^{-is\theta } \mathcal{E}_{s}(z) \right\}, 
\label{eq96} 
\end{equation}
and 
\begin{equation}
 d_{n}(u) =\left[ \exp \left\{ \sum\limits_{s=1}^{n-1}{\ \dfrac{ M_{s}}{\left\vert u\right\vert ^{s}}}\right\} +\exp \left\{ \sum\limits_{s=1}^{n-1}{\dfrac{N_{s}}{\left\vert u\right\vert ^{s}}} \right\} \right] e_{n}(u) \left\{ 1+\dfrac{e_{n}(u) }{2{ \left\vert u\right\vert ^{n}}}\right\} ^{2}, 
\label{d2m+1}
\end{equation}
 where $e_{n}(u) =\mathcal{O}(1) $ as $ u\rightarrow \infty $ and is given by 
\begin{multline}
 e_{n}(u) =\left\vert u\right\vert ^{n}\delta_{n}(u) +\omega_{n}(u) \exp \left\{ {\left\vert u\right\vert ^{-1}\varpi_{n}(u) +\left\vert u\right\vert ^{-n}\omega_{n}(u) }\right\} \\ +\gamma_{n}(u,\rho)\exp \left\{ {\left\vert u\right\vert^{-1}\beta_{n}(u,\rho)+\left\vert u\right\vert^{-n}\gamma_{n}(u,\rho)}\right\} . 
\label{en} 
\end{multline}

In addition, let $\tilde{M}_{s}$ and $\tilde{N}_{s}$ be given by (\ref{eq96}), except with $\mathcal{E}_{s}$ replaced by $\tilde{\mathcal{E}}_{s}$. Next let 
\begin{equation}
 \tilde{d}_{n}(u) =\left[ \exp \left\{ \sum\limits_{s=1}^{n-1}{ \dfrac{\tilde{M}_{s}}{\left\vert u\right\vert ^{s}}}\right\} +\exp \left\{ \sum\limits_{s=1}^{n-1}{\dfrac{\tilde{N}_{s}}{\left\vert u\right\vert ^{s}} }\right\} \right] \tilde{e}_{n}(u) \left\{ 1+\dfrac{\tilde{e}_{n}\left( {u }\right) }{2{\left\vert u\right\vert^{n}}}\right\} ^{2}, 
\label{d2m+2} 
\end{equation}
 where $\tilde{e}_{n}(u)$ is given by (\ref{en}) with $\gamma_{n}(u,\rho)$ and $\beta_{n}(u,\rho)$ replaced by $\tilde{\gamma}_{n}(u,\rho)$ and $\tilde{\beta}_{n}(u,\rho)$, respectively (see (\ref{eq28a}) and (\ref{eq29a})).
 
Let 
\begin{equation}
l_{0}(z) =\frac{4r_{0} K(k)}{\left\vert z-z_{0}\right\vert +r_{0}}, 
\label{eq50} 
\end{equation}
where 
\begin{equation}
k=\frac{2\sqrt{r_{0}\left\vert z-z_{0}\right\vert }}{\left\vert z-z_{0}\right\vert +r_{0}}, 
\label{defk} 
\end{equation}
and ${K}\left( k\right) $ is the complete elliptic integral of the first kind defined by \cite[\S 19.2(ii)]{NIST:DLMF} 
\begin{equation}
{K}\left( k\right) =\int\limits_{0}^{\pi /2}{\dfrac{d\tau }{\sqrt{ 1-k^{2}\sin ^{2}\left( \tau \right) }}}=\int\limits_{0}^{1}{\dfrac{dt}{\sqrt{ \left( 1-t^{2}\right) \left( 1-k^{2}t^{2}\right) }}\ }\left( 0\leq k<1\right) . 
\label{Kelliptic} 
\end{equation}
From \cite{Dunster:2020:SEB} it was shown that
\begin{equation}
l_{0}(z)=\oint_{\left\vert t-z_{0}\right\vert =r_{0}}\left\vert {\dfrac{dt}{t-z}}
\right\vert.
\label{l0}
\end{equation}

We now define certain functions that will also appear in our error bounds. These are new, and did not appear in \cite{Dunster:2020:SEB}. Let $ \mathcal{G}_{m,2s+1}(z) $ ($s=m,m+1,m+2,\cdots $) be defined by the convergent expansion 
\begin{multline}
 \dfrac{1}{\left\{ f(z) \zeta (z) \right\} ^{1/4}} \left[ {\exp \left\{ \sum\limits_{s=1}^{2m}{\dfrac{\mathcal{E}_{s}(z) }{u^{s}}}\right\} }\left( \exp \left\{ \sum \limits_{s=2m+1}^{2m+2r+1}{\dfrac{\mathcal{E}_{s}(z) }{u^{s}}} \right\} {-1}\right) \right. \\ \left. -{\exp \left\{ {\sum\limits_{s=1}^{2m}}(-1) ^{s}{\dfrac{ \mathcal{E}_{s}(z) }{u^{s}}}\right\} }\left( \exp \left\{ \sum\limits_{s=2m+1}^{2m+2r+1}(-1) ^{s}{\dfrac{\mathcal{E}_{s}(z) }{u^{s}}}\right\} {-1}\right) \right] =\sum\limits_{s=m}^{\infty }\dfrac{\mathcal{G}_{m,2s+1}(z)}{u^{2s+1}}.
\label{2.1} 
\end{multline}
Thus from (\ref{1.9}) and (\ref{2.1}) we have 
\begin{equation}
\dfrac{\varepsilon_{2m+1,r}(u,z) }{\left\{ f(z) \zeta (z) \right\} ^{1/4}}=\sum\limits_{s=m}^{\infty }{\dfrac{ \mathcal{G}_{m,2s+1}(z) }{u^{2s+1}}}+\frac{\varepsilon_{2m+2r+2}(u,z) }{\left\{ f(z) \zeta (z) \right\} ^{1/4}}.
\label{2.4} 
\end{equation}

We remark that $\mathcal{G}_{m,2s+1}(z) $ depend on $r$ for $ s\geq p+1$, where $p$ is the largest integer such that $2p+1\leq 2m+r$, i.e. $p=\left\lfloor m+\frac{1}{2}\left( r-1\right) \right\rfloor $. However we suppress this dependence. The first term is independent of $r$ for $r\geq 1$ and is given by 
\begin{equation} 
\mathcal{G}_{m,2m+1}(z) =\dfrac{2\mathcal{E}_{2m+1}(z) }{\left\{ f(z) \zeta (z) \right\} ^{1/4}}, 
\label{2.2} 
\end{equation}
and second term is independent of $r$ for $r\geq 3$ and is given by 
\begin{multline} 
\mathcal{G}_{m,2m+3}(z) =\dfrac{1}{\left\{ f(z) \zeta (z) \right\}^{1/4}} 
\left\{ 2 \mathcal{E}_{2m+3}(z) +2\mathcal{E}_{1}(z) \mathcal{E}_{2m+2}(z) \right.
\\ \left. +2\mathcal{E}_{2}(z) \mathcal{E}_{2m+1}(z) 
+ \mathcal{ E}_{1}^{2}(z) \mathcal{E}_{2m+1}(z) \right\}, 
\label{2.3} 
\end{multline}
and so on. 

Similarly $\mathcal{\tilde{G}}_{m,2s}(z) $ ($ s=m,m+1,m+2,\cdots $) are defined by 
\begin{multline}
 \left\{ \dfrac{\zeta (z) }{f(z) }\right\} ^{1/4} \left[ {\exp \left\{ \sum\limits_{s=1}^{2m+1}{\dfrac{\mathcal{\tilde{E}}_{s}(z) }{u^{s}}}\right\} }\left( \exp \left\{ \sum\limits_{s=2m+2}^{2m+2r+1}{\dfrac{\mathcal{\tilde{E}}_{s}(z) }{u^{s}}}\right\} {-1}\right) \right. \\ \left. +{\exp \left\{ {\sum\limits_{s=1}^{2m+1}}(-1) ^{s}{ \dfrac{\mathcal{\tilde{E}}_{s}(z) }{u^{s}}}\right\} }\left( \exp \left\{ \sum\limits_{s=2m+2}^{2m+2r+1}(-1) ^{s}{\dfrac{ \mathcal{\tilde{E}}_{s}(z) }{u^{s}}}\right\} {-1}\right) \right] =\sum\limits_{s=m+1}^{\infty }{\dfrac{\mathcal{\tilde{G}}_{m,2s}(z) }{u^{2s}}}. 
\label{2.5} 
\end{multline}

Thus from (\ref{1.10}) and (\ref{2.5}) we have 
\begin{equation}
\left\{ \dfrac{\zeta (z) }{f(z) }\right\} ^{1/4} \tilde{\varepsilon}_{2m+2,r}(u,z) =\sum\limits_{s=m+1}^{\infty }{ \dfrac{\mathcal{\tilde{G}}_{m,2s}(z) }{u^{2s}}}+\left\{ \dfrac{ \zeta (z) }{f(z) }\right\} ^{1/4}\tilde{\varepsilon}_{2m+2r+2}(u,z). 
\label{2.8} 
\end{equation}
These coefficients depend on $r$ for $s\geq \tilde{p}+1$, where $\tilde{p}$ is the largest integer such that $2\tilde{p}\leq 2m+r$, i.e. $\tilde{p} =\left\lfloor m+\frac{1}{2}r\right\rfloor $. Again we suppress any $r$ dependence. Here for $r\geq 2$ the first term is independent of $r$ and is given by 
\begin{equation}
\mathcal{\tilde{G}}_{m,2m+2}(z) =2\left\{ \dfrac{\zeta (z) }{f(z) }\right\} ^{1/4}\tilde{\mathcal{E}}_{2m+2}(z),
\label{2.6} 
\end{equation}
and for $r\geq 4$ the second term is independent of $r$ and is given by 
\begin{multline} 
\mathcal{\tilde{G}}_{m,2m+4}(z) =\left\{ \dfrac{\zeta (z) }{f(z) }\right\} ^{1/4}\left\{ 2\tilde{\mathcal{E}}_{2m+4}(z) +2\tilde{\mathcal{E}}_{1}(z) \tilde{ \mathcal{E}}_{2m+3}(z) \right.
\\ \left. +2\tilde{\mathcal{E}}_{2}(z) \tilde{\mathcal{E}}_{2m+2}(z) +\tilde{\mathcal{E}}_{1}^{2}(z) \tilde{\mathcal{E}}_{2m+2}(z) \right\}. 
\label{2.7} 
\end{multline}

Our main error bound theorem uses the following generalization of Cauchy's integral formula for meromorphic functions. In particular, this will allow us evaluate certain contour integrals appearing in the error bounds in a numerically satisfactory way. For a proof see \cite[Thm. 9]{Dunster:2020:ASI}.

\begin{theorem}
\label{thm:cauchy}

Let $C$ be a positively orientated simple loop in the $z$ plane, and $G(z)$ be a function that is analytic in the open region enclosed by the path and continuous on its closure, except for a pole of arbitrary order $p$ at an interior point $z=z_{0}$. Let $\left\{ g_{j}\right\}_{j=-\infty }^{\infty }$ be the Laurent coefficients of $G(z)$ at $z=z_{0}$, so that for some $r_{0}>0$ 
\begin{equation}
G(z) =\sum_{j=-p}^{\infty }g_{j}\left( z-z_{0}\right) ^{j}\
\left( 0<\left\vert z-z_{0}\right\vert <r_{0}\right) ,  \label{2.9}
\end{equation}
and let $G^{\ast}(z)$ denote the regular (or analytic) part of $G(z)$ at $z=z_{0}$, given by 
\begin{equation}
G^{\ast}(z) =\sum_{j=0}^{\infty }g_{j}\left( z-z_{0}\right)
^{j}\ \left( 0\leq \left\vert z-z_{0}\right\vert <r_{0}\right) .
\label{2.10}
\end{equation}
Then for all $z$ lying inside $C$ 
\begin{equation}
\oint_{C}\frac{G(t)}{t-z}dt{=2\pi iG}^{\ast}(z).  
\label{2.11}
\end{equation}

\end{theorem}

Each $\mathcal{G}_{m,n}(z) $ has a pole at $z_{0}$, and in accord with (\ref{2.10}) we define $\mathcal{G}_{m,n}^{\ast }(z) $ as its regular part at that pole. Likewise let $\tilde{\mathcal{G}}_{m,n}^{\ast }(z) $ be the regular part of $\tilde{\mathcal{G}}_{m,n}(z)$ at $z_{0}$. 

We now state our main theorem.

\begin{theorem}
\label{Thm3.2}
Assume \cref{hyp:1} and let $z$ lie in the interior of the circle $\Gamma$. Then for each positive integer $m$ and $r$ the differential equation (\ref{eq1}) has solutions (\ref{wjs}) with 
\begin{multline} 
\mathcal{A}_{2m+2}(u,z) =\dfrac{1}{2\pi i}\oint_{\left\vert t-z_{0}\right\vert =r_{0}}\exp \left\{ \sum\limits_{s=1}^{m}{\dfrac{\tilde{ \mathcal{E}}_{2s}(t) }{u^{2s}}}\right\} \\ \times \cosh \left\{ \sum\limits_{s=0}^{m}{\dfrac{\tilde{\mathcal{E}}_{2s+1}(t) }{ u^{2s+1}}}\right\} \left\{ \frac{\zeta (t) }{f(t) } \right\} ^{1/4}\dfrac{dt}{t-z}+\frac{1}{2}\tilde{\kappa}_{2m+2,r}(u,z),
\label{2.18} 
\end{multline}
where 
\begin{equation}
\left\vert \tilde{\kappa}_{2m+2,r}(u,z) \right\vert \leq \sum\limits_{s=m+1}^{\infty }{\dfrac{\left\vert \mathcal{\tilde{G}}_{m,2s}^{\ast }(z) \right\vert }{\left\vert u\right\vert ^{2s}}}+ \dfrac{\tilde{\Upsilon}\tilde{d}_{2m+2r+2}(u) l_{0}(z) }{2\pi\left\vert u\right\vert ^{2m+2r+2}}, 
\label{2.19} 
\end{equation}
and 
\begin{multline} 
\mathcal{B}_{2m+1}(u,z) =\dfrac{1}{2\pi iu^{1/3}} \oint_{\left\vert t-z_{0}\right\vert =r_{0}}\exp \left\{ \sum\limits_{s=1}^{m}{\dfrac{\mathcal{E}_{2s}(t) }{u^{2s}}} \right\} 
\\ \times \sinh \left\{ \sum\limits_{s=0}^{m-1}{\dfrac{\mathcal{E}_{2s+1}(t) }{u^{2s+1}}}\right\} \dfrac{dt}{\left\{ f(t) \zeta (t) \right\} ^{1/4}(t-z) }+\frac{ \kappa_{2m+1,r}(u,z) }{2u^{1/3}}, 
\label{2.20} 
\end{multline}
where 
\begin{equation}
\left\vert \kappa_{2m+1,r}(u,z) \right\vert \leq \sum\limits_{s=m}^{\infty }{\dfrac{\left\vert \mathcal{G}_{m,2s+1}^{\ast }(z) \right\vert }{\left\vert u\right\vert ^{2s+1}}}+\dfrac{ d_{2m+2r+2}(u) l_{0}(z) }{2\pi \Upsilon \left\vert u\right\vert ^{2m+2r+2}}.
\label{2.21} 
\end{equation}
\end{theorem}

\begin{proof}
We have from (\ref{1.1}) and (\ref{2.18}) 
\begin{equation}
\tilde{\kappa}_{2m+2,r}(u,z) =\dfrac{1}{2\pi i}\oint_{\left\vert t-z_{0}\right\vert =r_{0}}\left\{ \frac{\zeta (t) }{f(t) }\right\} ^{1/4}{\dfrac{\tilde{\varepsilon}_{2m+2,r}(u,t) dt}{t-z}}.
\label{2.22} 
\end{equation}
Now from this and (\ref{2.8}) we have 
\begin{multline} 
\tilde{\kappa}_{2m+2,r}(u,z) =\dfrac{1}{2\pi i} \sum\limits_{s=m+1}^{\infty }{\dfrac{1}{u^{2s}}}\oint_{\left\vert t-z_{0}\right\vert =r_{0}}{\dfrac{\mathcal{\tilde{G}}_{m,2s}(t) dt}{t-z}}
\\ +\dfrac{1}{2\pi i}\oint_{\left\vert t-z_{0}\right\vert =r_{0}}\left\{ \dfrac{\zeta (t) }{f(t) }\right\}^{1/4}{\dfrac{\tilde{\varepsilon}_{2m+2r+2}(u,t) dt}{t-z}}.
\label{2.23} 
\end{multline}

Next from \cite{Dunster:2020:SEB} 
\begin{equation}
\tilde{\kappa}_{2m+2}(u,z) =\dfrac{1}{2\pi i}\oint_{\left\vert t-z_{0}\right\vert =r_{0}}\left\{ \frac{\zeta (t) }{f(t) }\right\} ^{1/4}{\dfrac{\tilde{\varepsilon}_{2m+2}(u,t) dt}{t-z}}. 
\label{2.24} 
\end{equation}
Hence from (\ref{2.11}), (\ref{2.23}), and (\ref{2.24}) with $m$ replaced by $m+r$ we have 
\begin{equation}
\tilde{\kappa}_{2m+2,r}(u,z) =\sum\limits_{s=m+1}^{\infty }{ \dfrac{\mathcal{\tilde{G}}_{m,2s}^{\ast }(z) }{u^{2s}}+}\tilde{ \kappa}_{2m+2r+2}(u,z). 
\label{2.25} 
\end{equation}
Now from \cite{Dunster:2020:SEB} 
\begin{equation}
\left\vert \tilde{\kappa}_{2m+2r+2}(u,z) \right\vert \leq \dfrac{ \tilde{\Upsilon}\tilde{d}_{2m+2r+2}(u) l_{0}(z) }{ 2\pi \left\vert u\right\vert ^{2m+2r+2}}, 
\label{2.26} 
\end{equation}
and (\ref{2.19}) then follows from (\ref{2.25}). 

The bound (\ref{2.21}) is proved similarly from the relation
\begin{multline} 
\kappa_{2m+1,r}(u,z) =\dfrac{1}{2\pi i} \sum\limits_{s=m}^{\infty }\dfrac{1}{u^{2s+1}}\oint_{\left\vert t-z_{0}\right\vert =r_{0}}{\dfrac{\mathcal{G}_{m,2s+1}(t) dt}{t-z}}
\\ +\dfrac{1}{2\pi i}\oint_{\left\vert t-z_{0}\right\vert =r_{0}}\frac{\varepsilon_{2m+2r+2}(u,t) dt }{\left\{ f(t) \zeta (t) \right\} ^{1/4}(t-z)}.
\label{2.23a} 
\end{multline}

\end{proof}

\subsection{Error bounds for the series appearing in (\ref{2.19}) and (\ref{2.21})}

In evaluating the bounds from (\ref{2.19}) and (\ref{2.21}) one of course would just compute the first few terms of the convergent series involving the $\mathcal{\tilde{G}}_{m,2s}^{\ast }(z)$ and $\mathcal{G}_{m,2s+1}^{\ast }(z)$ coefficients. For completeness we include here bounds for the remainders of such truncated series, even though in practice we would not usually compute these.

This is an exercise involving Maclaurin series error bounds. In general, let $G(w)$ be analytic in an open set containing the disk $\{w:|w|\leq a\}$. We shall use the well-known result that
\begin{equation}
G(w)=\sum_{s=0}^{n}\frac{G^{(s)}(0)}{s!}w^{s}+R_{n+1}\left( w\right),
\label{G1}
\end{equation}
where for $0\leq |w|<a$
\begin{equation}
R_{n+1}\left( w\right) =\frac{w^{n+1}}{2\pi i}\oint_{\left\vert v\right\vert
=a}\frac{G(v)dv}{v^{n+1}(v-w)},
\label{G2}
\end{equation}
and hence 
\begin{equation}
\left\vert R_{n+1}\left( w\right) \right\vert \leq \sup_{|w|=a}\left\vert
G(w)\right\vert \frac{|w|^{n+1}}{a^{n}(a-|w|)}\ (0\leq |w|<a).
\label{G3}
\end{equation}

Let us apply this to the series (\ref{2.1}), with $w=1/u$ regarded as a small complex variable. Fix $z\in \Gamma$ and let
\begin{multline}
G_{2m+1}(w,z)=\dfrac{1}{\left\{ f(z)\zeta (z)\right\} ^{1/4}}\left[ 2{\exp
\left\{ \sum\limits_{s=1}^{m+r}\mathcal{E}_{2s}(z)w^{2s}\right\} \sinh
\left\{ \sum\limits_{s=0}^{m+r}\mathcal{E}_{2s+1}(z)w^{2s+1}\right\} }
\right.  \\
\left. -2{\exp \left\{ \sum\limits_{s=1}^{m}\mathcal{E}_{2s}(z)w^{2s}\right
\} \sinh \left\{ {\sum\limits_{s=0}^{m-1}}\mathcal{E}_{2s+1}(z)w^{2s+1}
\right\} }\right].  
\label{G4}
\end{multline}
Then $G_{2m+1}(w,z)$ is entire in $w$, so that for $0\leq |w|<\infty $ it possesses the Maclaurin expansion
\begin{equation}
G_{2m+1}(w,z)=\sum\limits_{s=m}^{\infty }\mathcal{G}_{m,2s+1}(z)w^{2s+1}.
\label{G5}
\end{equation}
Note that in relation to (\ref{G1}) only odd powers of $w$ appear, and of these the first $m$ terms are identically zero, but this obviously does not affect the validity of the bound (\ref{G3}).

We now apply (\ref{G1}) - (\ref{G3}), and in these we can choose any positive value of $a$. However, it must not be too small on account of the $a^{n}$ term in the denominator in the bound (\ref{G3}), but also must not be too large on account of the supremum appearing in this bound. 

Given these considerations our choice is given by $a=1/u_{m}$, where
\begin{equation}
u_{m}=\left( E_{2m+2r+1}\right) ^{1/(2m+2r+1)},
\label{G6}
\end{equation}
in which
\begin{equation}
E_{s}=\underset{z\in \Gamma }{\sup }\left\vert \mathcal{E}_{s}(z)\right\vert.
\label{G7}
\end{equation}

Now let us apply (\ref{G1}) - (\ref{G3}). We truncate the series (\ref{G5}) at $s=N$ for arbitrary $N\geq m$, yielding
\begin{equation}
G_{2m+1}(w,z)=\sum\limits_{s=m}^{N}\mathcal{G}
_{m,2s+1}(z)w^{2s+1}+R_{m,2N+3}(w,z),
\label{G8}
\end{equation}
where for $|w|<1/u_{m}$
\begin{equation}
R_{m,2N+3}(w,z)=\frac{w^{2N+3}}{2\pi i}\oint_{\left\vert v\right\vert =w_{m}}
\frac{G_{2m+1}(v,z)dv}{v^{2N+3}(v-w)}.
\label{G9}
\end{equation}

Next for $z\in \Gamma $ we have from (\ref{zeds}), (\ref{G4}) and (\ref{G7})
\begin{equation}
\sup_{|w|=1/u_{m}}\left\vert G_{2m+1}(w,z) \right\vert \leq G_{m},
\label{G10}
\end{equation}
where
\begin{multline}
G_{m}=\dfrac{2}{\Upsilon }{\exp \left\{ \sum\limits_{s=1}^{m+r}\frac{E_{2s}}{
u_{m}^{2s}}\right\} \sinh \left\{ \sum\limits_{s=0}^{m+r}\frac{E_{2s+1}}{
u_{m}^{2s+1}}\right\} } \\
+\dfrac{2}{\Upsilon }{\exp \left\{ \sum\limits_{s=1}^{m}\frac{E_{2s}}{
u_{m}^{2s}}\right\} \sinh }\left\{ \sum_{s=0}^{m-1}\frac{E_{2s+1}}{
u_{m}^{2s+1}}\right\}.
\label{G11}
\end{multline}
Hence for $|w|<1/u_{m}$ we have from (\ref{G9})
\begin{equation}
\left\vert R_{m,2N+3}(w,z)\right\vert \leq \frac{G_{m}\left( u_{m}|w|\right)
^{2N+3}}{\left( 1-u_{m}|w|\right) }. \label{G12}
\end{equation}

Typically for large $s$ we have $E_{s}\sim k_{s}l^{s}s!$ for some
positive constant $l$ and slowly varying $k_{s}$ which is ${\mathcal{O}}(1)$ (c.f. \cref{thm:as}). If so, then by Stirling's formula we find for $1\leq n\leq 2m+2r+1$
\begin{equation}
\frac{E_{n}}{u_{m}^{n}}=\left( \frac{n}{2m+2r+1}\right) ^{n}{\mathcal{O}}
(1)\ (m\rightarrow \infty ).
\label{G13}
\end{equation}
Hence all sums appearing in (\ref{G11}) are ${\mathcal{O}}\left( 1\right) $ and
hence so is $G_{m}$. Thus, for $0\leq |w|\leq \delta /u_{m}$ where $0<\delta
<1$ 
\begin{equation}
R_{m,2N+3}(w,z)=\left( E_{2m+2r+1}\right) ^{(2N+3)/(2m+2r+1)}{\mathcal{O}}
\left( w^{2N+3}\right) \ (m\rightarrow \infty ).
\label{G14}
\end{equation}
Note if $N<m+r$ then
\begin{equation}
R_{m,2N+3}(w,z)=E_{2m+2r+1}{\mathcal{O}}\left( w^{2N+3}\right) \
(m\rightarrow \infty ).
\label{G15}
\end{equation}
This is the reason for our choice (\ref{G6}).

Bringing everything together, from (\ref{2.1}), (\ref{G4}), (\ref{G9}) and \cref{thm:cauchy} we have for integer $N\geq m$

\begin{multline}
\dfrac{1}{2\pi i} \sum\limits_{s=m}^{\infty }\dfrac{1}{u^{2s+1}}\oint_{\left\vert t-z_{0}\right\vert =r_{0}}{\dfrac{\mathcal{G}_{m,2s+1}(t) dt}{t-z}} \\
=\sum\limits_{s=m}^{N}\dfrac{ \mathcal{G}_{m,2s+1}^{\ast }(z) }{ u^{2s+1}}+\dfrac{1}{2\pi i}\oint_{\left\vert t-z_{0}\right\vert =r_{0}}{\dfrac{R_{m,2N+3}(u^{-1},t) dt}{t-z}}.
\label{G15a}
\end{multline}
Hence on referring to (\ref{l0}), (\ref{2.21}), (\ref{2.23a}) and (\ref{G12}) we have our main result that
\begin{multline}
\left\vert \kappa _{2m+1,r}(u,z)\right\vert \leq \sum\limits_{s=m}^{N}{
\dfrac{\left\vert \mathcal{G}_{m,2s+1}^{\ast }(z)\right\vert }{\left\vert
u\right\vert ^{2s+1}}} \\
+\left( \frac{u_{m}}{|u|}\right) ^{2N+3}\frac{G_{m}{l}_{0}{(z)}}{2\pi \left(
1-u_{m}|u|^{-1}\right) }+\dfrac{d_{2m+2r+2}(u)l_{0}(z)}{2\pi \Upsilon
\left\vert u\right\vert ^{2m+2r+2}}.
\label{G16}
\end{multline}

Similarly let
\begin{equation}
\tilde{E}_{s}=\underset{z\in \Gamma }{\sup }\left\vert \mathcal{\tilde{E}}
_{s}(z)\right\vert,  
\label{G18}
\end{equation}
\begin{equation}
\tilde{u}_{m}=\left( \tilde{E}_{2m+2r+1}\right) ^{1/(2m+2r+1)},
\label{G17}
\end{equation}
and
\begin{multline}
\tilde{G}_{m}=2\tilde{\Upsilon}{\exp }\left\{ {\sum\limits_{s=1}^{m+r}}\frac{\tilde{E}_{2s}}{\tilde{u}_{m}^{2s}}\right\} \cosh {\left\{
\sum\limits_{s=0}^{m+r}\frac{\tilde{E}_{2s+1}}{\tilde{u}_{m}^{2s+1}}\right\} 
} \\
+2\tilde{\Upsilon}{\exp \left\{ \sum\limits_{s=1}^{m}\frac{{\tilde{E}_{2s}}}{
\tilde{u}_{m}^{2s}}\right\} \cosh }\left\{ \sum_{s=0}^{m}\frac{\tilde{E}
_{2s+1}}{\tilde{u}_{m}^{2s+1}}\right\} .  \label{G22}
\end{multline}
Then for integer $N\geq m+1$
\begin{multline}
\left\vert \tilde{\kappa}_{2m+2,r}(u,z)\right\vert \leq
\sum\limits_{s=m+1}^{N}{\dfrac{\left\vert \mathcal{\tilde{G}}_{m,2s}^{\ast
}(z)\right\vert }{\left\vert u\right\vert ^{2s}}} \\
+\left( \frac{\tilde{u}_{m}}{|u|}\right) ^{2N+2}\frac{\tilde{G}_{m}{l}_{0}{(z)
}}{2\pi \left( 1-\tilde{u}_{m}|u|^{-1}\right) }+\dfrac{\tilde{\Upsilon}
\tilde{d}_{2m+2r+2}(u)l_{0}(z)}{2\pi \left\vert u\right\vert ^{2m+2r+2}}.
\label{G23}
\end{multline}

\section{Bessel functions of large order} 
\label{sec4} 
Following \cite{Dunster:2020:SEB} in the equation (\ref{eq1}) we take $u=\nu $ and the functions $f(z)$ and $g(z)$ are given by 
\begin{equation}
f(z)=\frac{1-z^{2}}{z^{2}},\ g(z)={-\frac{1}{4z^{2}}},
\label{3.1a} 
\end{equation}
and from (\ref{xi})
\begin{equation}
\xi =\frac{2}{3}\zeta ^{3/2}=\ln \left\{ {\frac{1+\left( {1-z^{2}}\right)
^{1/2}}{z}}\right\} -\left( {1-z^{2}}\right) ^{1/2}.
\label{xiBessel}
\end{equation}

The coefficients (\ref{eq7}) are given by 
\begin{equation}
\hat{E}_{s}(z) =\int_{z}^{\infty }t^{-1}\left( {1-t^{2}} \right) ^{1/2}\hat{F}{(t) dt}\quad \left( {s=1,2,3,\cdots } \right) , 
\label{3.1} 
\end{equation}
where 
\begin{equation} \hat{F}_{1}(z)=\frac{z^{2}(z^{2}+4)}{8(z^{2}-1)^{3}},\,\hat{F}_{2}(z)=\frac{z}{2\left( 1-z^{2}\right) ^{1/2}}\hat{F}_{1}^{\prime }(z), 
\label{3.2} 
\end{equation}
and 
\begin{equation} \hat{F}_{s+1}(z)=\frac{z}{2\left( 1-z^{2}\right) ^{1/2}}\hat{F}_{s}^{\prime }(z)-\frac{1}{2}\sum_{j=1}^{s-1}\hat{F}_{j}(z)\hat{F}_{s-j}(z)\quad \left( {s=2,3,\cdots }\right) . 
\label{recuFs} 
\end{equation}

Let $C_{2s+1}$ are the coefficients in the Stirling asymptotic series 
\begin{equation}
\Gamma(\nu)\sim \left( 2\pi \right) ^{1/2}e^{-\nu}\nu^{{\nu -(1/2)} }\exp \left\{ \sum_{j=0}^{\infty }\frac{C_{2j+1}}{\nu ^{2j+1}}\right\} \ (\nu\rightarrow \infty ). 
\label{3.4} 
\end{equation}
with ${C_{2j}=0}$ ($j=1,2,3,\cdots $).

From \cite{Dunster:2020:SEB} we have the exact expressions 
\begin{multline}
 \mathcal{A}_{2m+2}(\nu,z) =\pi ^{1/2}{e^\nu\nu^{-{\nu +(5/6)}}\Gamma (\nu) }\exp \left\{ -\sum\limits_{j=0}^{m+r-1} {\ \dfrac{C_{2j+1}}{\nu^{2j+1}}}\right\} \\ \times z^{1/2}\left\{ {e^{\pi i/6}\mathrm{Ai}_{-1}^{\prime }\left( {\nu ^{2/3}\zeta }\right) J_\nu(\nu z) -\tfrac{1}{2}i\mathrm{Ai}_{0}^{\prime }\left(\nu ^{2/3}\zeta \right) H_\nu^{(1) }(\nu z) }\right\}, 
\label{3.5} 
\end{multline}
 and 
\begin{multline}
 \mathcal{B}_{2m+1}(\nu,z) =\pi ^{1/2}e^{\nu}{\nu^{-{\nu +(5/6)}}}\Gamma (\nu) \exp \left\{ -\sum\limits_{j=0}^{m+r-1} {\ \dfrac{C_{2j+1}}{\nu^{2j+1}}}\right\} \\ \times z^{1/2}\left\{ \tfrac{1}{2}i{\mathrm{Ai}_{0}\left( {\nu ^{2/3}\zeta } \right) }H_\nu^{(1) }(\nu z) {-e^{\pi i/6} \mathrm{\ Ai}_{-1}\left(\nu ^{2/3}\zeta \right) }J_\nu(\nu z) \right\}.
\label{3.6} 
\end{multline}

On applying \cref{thm:far} we have 
\begin{multline}
 \mathcal{A}_{2m+2}(\nu,z) =\left( {\dfrac{z^{2}\zeta }{1-z^{2} }}\right) ^{1/4} \\ \times \left[ \exp \left\{\sum\limits_{s=1}^{m}{\dfrac{\mathcal{\tilde{E}}_{2s}(z) }{\nu^{2s}}}\right\} \cosh \left\{ \ \sum\limits_{s=0}^{m}{\dfrac{\mathcal{\tilde{E}}_{2s+1}(z) }{{\ \nu }^{2s+1}}}\right\} +\frac{1}{2}\tilde{\varepsilon}_{2m+2,r}(\nu,z) \right],
\label{3.7} 
\end{multline}
 and 
\begin{multline}
 \mathcal{B}_{2m+1}(\nu,z) =\dfrac{1}{\nu^{1/3}}\left\{ {\ \dfrac{z^{2}}{\zeta \left( 1-z^{2}\right) }}\right\} ^{1/4} \\ \times \left[ {\exp \left\{ \sum\limits_{s=1}^{m}{\dfrac{\mathcal{E}_{2s}(z) }{\nu^{2s}}}\right\} \sinh \left\{\ \sum\limits_{s=0}^{m-1}{\dfrac{\mathcal{E}_{2s+1}(z) }{\nu ^{2s+1}}}\right\} }+\frac{1}{2}\varepsilon_{2m+1,r}(\nu,z) \right].
\label{3.8} 
\end{multline}
 In these 
\begin{multline}
 \left\vert \tilde{\varepsilon}_{2m+2,r}(\nu,z) \right\vert \leq {\exp \left\{ \sum\limits_{s=1}^{2m+1}{\dfrac{{\mathrm{Re\ }}\mathcal{ \tilde{E}}_{s}(z) }{\nu ^{s}}}\right\} }\left\vert \exp \left\{ \sum\limits_{s=2m+2}^{2m+2r+1}{\dfrac{\mathcal{\tilde{E}}_{s}(z) }{\nu ^{s}}}\right\} {-1}\right\vert \\ +{\exp \left\{ {\sum\limits_{s=1}^{2m+1}}(-1) ^{s}{\dfrac{{ \mathrm{Re\ }}\mathcal{\tilde{E}}_{s}(z) }{\nu ^{s}}}\right\} } \left\vert \exp \left\{ \sum\limits_{s=2m+2}^{2m+2r+1}(-1) ^{s}{ \dfrac{\mathcal{\tilde{E}}_{s}(z) }{\nu ^{s}}}\right\} {-1} \right\vert +\frac{\tilde{d}_{2m+2r+2}(\nu,z) }{{\nu ^{2m+2r+2}} },
\label{3.9} 
\end{multline}
 and 
\begin{multline}
 \left\vert \varepsilon_{2m+1,r}(\nu,z) \right\vert \leq {\exp \left\{ \sum\limits_{s=1}^{2m}{\dfrac{{\mathrm{Re\ }}\mathcal{E}_{s}(z) }{\nu ^{s}}}\right\} }\left\vert \exp \left\{ {\sum \limits_{s=2m+1}^{2m+2r+1}{\dfrac{\mathcal{E}_{s}(z) }{\nu ^{s}}} }\right\} {-1}\right\vert \\ +{\exp \left\{ {\sum\limits_{s=1}^{2m}}(-1) ^{s}{\dfrac{{ \mathrm{Re\ }}\mathcal{E}_{s}(z) }{\nu ^{s}}}\right\} } \left\vert {\exp \left\{ \sum\limits_{s=2m+1}^{2m+2r+1}(-1) ^{s} {\dfrac{\mathcal{E}_{s}(z) }{\nu ^{s}}}\right\} {-1}} \right\vert +\frac{d_{2m+2r+2}(\nu,z) }{{\nu ^{2m+2r+2}}},
\label{3.10} 
\end{multline}
 where for ${\nu >0}$ and $z\in S_{0,-1}\cup S_{-1,0}$ 
\begin{multline}
 \tilde{d}_{2m+2r+2}(\nu,z) 
 \\=\exp \left\{\ \sum\limits_{s=1}^{2m+2r+1}{\dfrac{{\mathrm{Re\ }}\mathcal{\tilde{E}}_{s}(z) }{{\ \nu }^{s}}}\right\} \tilde{e}_{2m+2r+2,-1}\left( { \nu ,z}\right) \left\{ 1+\dfrac{\tilde{e}_{2m+2r+2,-1}(\nu,z) }{2{\nu^{2m+2r+2}}}\right\} ^{2} \\ +\exp \left\{ \sum\limits_{s=1}^{2m+2r+1}{\left( {-1}\right) ^{s}\dfrac{{ \mathrm{Re\ }}\mathcal{\tilde{E}}_{s}(z) }{\nu^{s}}}\right\} \tilde{e}_{2m+2r+2,0}(\nu,z) \left\{ 1+\dfrac{\tilde{e}_{2m+2r+2,0}(\nu,z) }{2{\nu^{2m+2r+2}}}\right\} ^{2},
\label{3.11} 
\end{multline}
 in which (for $j=0,-1$) 
\begin{multline}
 \tilde{e}_{2m+2r+2,j}(\nu,z)={\nu^{2m+2r+2}\delta }_{2m+2r+2,j}(\nu) \\ +\omega_{2m+2r+2,j}(\nu,z) \exp \left\{ {\nu^{-1}\varpi_{2m+2r+2,j}(\nu,z) +\nu^{-2m-2r-2}\omega_{2m+2r+2,j}(\nu,z) }\right\} \\ +\tilde{\gamma}_{2m+2r+2}\left( {\nu ,\xi }\right) \exp \left\{ {\nu^{-1} \tilde{\beta}_{2m+2r+2}\left( {\nu ,\xi }\right) +\nu^{-2m-2r-2}\tilde{ \gamma}_{2m+2r+2}(\nu,\xi)}\right\}. 
\label{3.12} 
\end{multline}
Here ${\delta }_{2m+2r+2,0}(\nu) =0$, 
\begin{equation}
\delta_{2m+2r+2,\pm 1}(\nu) =\left( \frac{1}{2\pi } \right) ^{1/2}\frac{e^\nu\Gamma (\nu) }{\nu^{\nu -(1/2)}}\exp \left\{ -\sum\limits_{j=0}^{m+r}{\frac{{C}_{2j+1}}{\nu ^{2j+1}}}\right\} -1,
\label{3.13} 
\end{equation}

\begin{multline}
 \omega_{2m+2r+2,0}(\nu,z) =2\int_{0}^{z}{\left\vert \dfrac{{ \hat{F}_{2m+2r+2}(t) }\left( 1-t^{2}\right)^{1/2}}{t}{dt} \right\vert } \\ +\sum\limits_{s=1}^{2m+2r+1}\dfrac{1}{\nu ^{s}}\int_{0}^{z}{\left\vert { \sum\limits_{k=s}^{2m+2r+1}}\dfrac{{{\hat{F}_{k}(t) \hat{F}_{s+2m+2r-k+1}(t) }}\left( 1-t^{2}\right) ^{1/2}}{t}{dt} \right\vert },
\label{3.14} 
\end{multline}
 
\begin{equation}
\varpi_{2m+2r+2,0}(\nu,z) =4\sum\limits_{s=0}^{2m+2r}\frac{1 }{{\nu ^{s}}}\int_{0}^{z}{\left\vert \frac{{\hat{F}_{s+1}(t) }\left( 1-t^{2}\right) ^{1/2}}{t}{dt}\right\vert}, 
\label{3.15} 
\end{equation}
and $\omega_{2m+2r+2,-1}(\nu,z)$ and $\varpi_{2m+2r+2,-1}(\nu,z)$ are the same except the lower limits of integration are $i\infty $ instead of $0$. The paths of integration can be taken as straight lines in both cases, in the latter case vertical lines from $z$ to infinity.

Similarly $d_{2m+2r+2}(\nu,z)$ is given by (\ref{3.11}) and (\ref{3.12}), except $\mathcal{\tilde{E}}_{s}(z)$, $\tilde{\gamma}_{2m+2r+2}(\nu,\xi)$ and $\tilde{\beta}_{2m+2r+2}(\nu,\xi)$ are replaced by $\mathcal{E}_{s}(z)$, $\gamma_{2m+2r+2}(\nu,\xi)$ and $\beta_{2m+2r+2}(\nu,\xi)$, respectively.

Near the turning point we apply \cref{Thm3.2}, and from this we have 
\begin{multline} 
\mathcal{A}_{2m+2}(\nu,z) =\dfrac{1}{2\pi i}\oint_{\left\vert t-z_{0}\right\vert =r_{0}}\exp \left\{ \sum\limits_{s=1}^{m}{\dfrac{\tilde{ \mathcal{E}}_{2s}(t) }{\nu ^{2s}}}\right\} 
\\ \times \cosh \left\{  \sum\limits_{s=0}^{m}{\dfrac{\tilde{\mathcal{E}}_{2s+1}(t) }{\nu ^{2s+1}}}\right\} \left\{ {\dfrac{t^{2}\zeta (t)}{1-t^{2}}}\right\} ^{1/4} \dfrac{dt}{t-z}+\frac{1}{2}\tilde{\kappa}_{2m+2,r}(\nu,z),
\label{3.15a} 
\end{multline}
where 
\begin{equation} 
\left\vert \tilde{\kappa}_{2m+2,r}(\nu,z) \right\vert \leq \sum\limits_{s=m+1}^{\infty }{\dfrac{\left\vert \mathcal{\tilde{G}}_{m,2s}^{\ast }(z) \right\vert }{\nu ^{2s}}}+\dfrac{\tilde{ \Upsilon}\tilde{d}_{2m+2r+2}(\nu) l_{0}(z) }{2\pi \nu ^{2m+2r+2}},
\label{3.15b} 
\end{equation}
and 
\begin{multline} 
\mathcal{B}_{2m+1}(\nu,z) 
=\dfrac{1}{2\pi i\nu ^{1/3}} 
\oint_{\left\vert t-z_{0}\right\vert =r_{0}}
\exp \left\{ 
\sum\limits_{s=1}^{m}\dfrac{  \mathcal{E}_{2s}(t) }{\nu ^{2s}} 
\right\} \\ 
\times \sinh \left\{ 
\sum\limits_{s=0}^{m-1} \dfrac{\mathcal{E}_{2s+1}(t) }{\nu ^{2s+1}}
\right\} 
\left\{ 
\dfrac{t^{2}}{\zeta (t)\left( 1-t^{2}\right) }
\right\}^{1/4}
\dfrac{dt}{t-z}+\frac{\kappa_{2m+1,r}(\nu,z) }{2\nu ^{1/3}},
\label{3.15c}
\end{multline}
where 
\begin{equation}
\left\vert \kappa_{2m+1,r}(\nu,z) \right\vert \leq \sum\limits_{s=m}^{\infty }{\dfrac{\left\vert \mathcal{G}_{m,2s+1}^{\ast }(z) \right\vert }{\nu ^{2s+1}}}+\dfrac{d_{2m+2r+2}(\nu) l_{0}(z) }{2\pi \Upsilon \nu ^{2m+2r+2}}.
\label{3.15d} 
\end{equation}
Here $d_{n}(\nu)$ and $\tilde{d}_{n}(\nu)$ are given by (\ref{d2m+1}) and (\ref{d2m+2}) with $u$ replaced by $\nu$.

On identifying with the standard Bessel functions we firstly have 
\begin{equation}
J_\nu(\nu z) =c_{m,0}(\nu) z^{-1/2}\left\{ \mathrm{Ai}\left( {\nu ^{2/3}\zeta }\right) \mathcal{A}_{2m+2}\left( {\nu ,z} \right) +\mathrm{Ai}^{\prime }\left( {\nu ^{2/3}\zeta }\right) \mathcal{B}_{2m+1}(\nu,z) \right\},
\label{3.16} 
\end{equation}
where 
\begin{multline}
 c_{m,0}(\nu) =\dfrac{2\pi ^{1/2}\nu^{\nu -(5/6)}}{e^\nu\Gamma (\nu) } \\ \times \left[ {\exp \left\{ -{\sum\limits_{j=0}^{m+r}}\dfrac{C_{2j+1}}{\nu ^{2j+1}}\right\} }+\frac{1}{2}\tilde{\varepsilon}_{2m+2,r}(\nu,0) -\frac{1}{2}\varepsilon_{2m+1,r}\left( \nu,0\right) \right] ^{-1}. 
\label{3.17} 
\end{multline}

 Next for the Hankel function we have 
\begin{multline} 
H_\nu^{(1)}(\nu z) =c_{m,-1}(\nu) z^{-1/2}\left\{ \mathrm{Ai}_{-1}\left( {\nu ^{2/3}\zeta }\right) \mathcal{A}_{2m+2}(\nu,z) \right.  \\ \left.  +\mathrm{Ai}_{-1}^{\prime }\left( {\nu ^{2/3}\zeta }\right) \mathcal{B}_{2m+1}(\nu,z) \right\}.
\label{3.18} 
\end{multline}
To find $c_{m,-1}(\nu) $ let $z\rightarrow i\infty $ and use 
\begin{equation}
H_\nu^{(1) }(\nu z) \sim \left( \frac{2}{\pi { \nu z}}\right) ^{1/2}\exp \left\{ i{\nu z-}\frac{1}{2}{\nu \pi i-}\frac{1}{ 4}{\pi i}\right\},
\label{3.19} 
\end{equation}
along with 
\begin{equation}
\xi =iz-\tfrac{1}{2}\pi i+\mathcal{O}\left( z^{-1}\right),
\label{3.20} 
\end{equation}

\begin{equation}
\mathrm{Ai}_{-1}\left( {\nu ^{2/3}\zeta }\right) \sim \frac{{e}^{-\pi i/6}e^{i\nu z-\frac{1}{2}\nu\pi i}}{2\pi ^{1/2}\nu^{1/6}\zeta ^{1/4}} \ \left( \nu\xi \rightarrow +\infty \right), 
\label{3.21} 
\end{equation}
and 
\begin{equation}
\mathrm{Ai}_{-1}^{\prime }\left( {\nu ^{2/3}\zeta }\right) \sim -\frac{e ^{-\pi i/6}\nu^{1/6}{\zeta }^{1/4}e^{i\nu z-\frac{1}{2}\nu\pi i}}{ 2\pi ^{1/2}}\ \left( \nu\xi \rightarrow +\infty \right). 
\label{3.22} 
\end{equation}
As a result we have as $z\rightarrow i\infty $ 
\begin{equation}
\mathcal{A}_{2m+2}(\nu,z) \sim \left( {-}\zeta \right) ^{1/4} \left\{ 1+\frac{1}{2}\tilde{\varepsilon}_{2m+2,r}\left( \nu,i\infty \right) \right\},
\label{3.23} 
\end{equation}
and 
\begin{equation}
\mathcal{B}_{2m+1}(\nu,z) \sim \dfrac{1}{2\nu^{1/3}} 
\left( -\ \dfrac{1}{\zeta }\right) ^{1/4}\varepsilon_{2m+1,r}\left( { \nu },i\infty \right).
\label{3.24} 
\end{equation}
Hence from (\ref{3.18}) - (\ref{3.24}) we find that 
\begin{equation}
c_{m,-1}(\nu) =\frac{2^{3/2} e^{-\pi i/3}}{\nu^{1/3}} \left[ 1+\frac{1}{2}\tilde{\varepsilon}_{2m+2,r}\left( \nu,i\infty \right) -\dfrac{1}{2}\varepsilon_{2m+1,r}\left( \nu,i\infty \right) \right] ^{-1}. 
\label{3.25} 
\end{equation}

\subsection{Computations away from the turning point}
\label{sec5}
In Tables \ref{table1}, \ref{table2} the error bounds given in (\ref{3.9}) and (\ref{3.10})
are compared with the true numerical errors obtained using (\ref{3.7}) and (\ref{3.8}) 
to approximate (\ref{3.5}) and (\ref{3.6}), respectively. In the comparisons, the values 
of $\nu$, $m$ and $r$ are fixed ($\nu=100$, $m=1$, $r=4$). 
We use Maple with
a large number of digits to evaluate (\ref{3.5}) and (\ref{3.6}).
A Gauss-Legendre quadrature with 30 nodes has been used to evaluate the integrals (\ref{3.14}) and (\ref{3.15})
appearing in the error bounds; the same quadrature has been applied (over a truncated interval) for the integrals for   
$\omega _{2m+2,-1}\left( {\nu ,z}\right) $ and $\varpi _{2m+2,-1}\left( {\nu ,z}\right) $.

\begin{table}
\label{table1}
$$
\begin{array}{lll}
\hline
z     & \mbox{True error} & \mbox{Error bound}   \\
  \hline
0.01 & 0.3494419980168294...  \times 10^{-11} & 0.3494419980326348... \times 10^{-11} \\
\hline
0.05 & 0.1308990430988040... \times 10^{-10} &  0.1308990431015310... \times 10^{-10}\\
\hline
0.15 &   0.5469585035500713... \times 10^{-10} &  0.5469585035736077... \times 10^{-10} \\
\hline
0.2 &     0.1089259283092040... \times 10^{-9} &   0.1089259283203728... \times 10^{-9} \\
\hline
0.25 &   0.2157034971499353... \times 10^{-9} &   0.2157034972175942... \times 10^{-9} \\
\hline
0.3 &     0.4236344565383917... \times 10^{-9} &   0.4236344570004239... \times 10^{-9} \\
\hline
0.35 &   0.8307697485641604... \times 10^{-9} &   0.8307697519296490... \times 10^{-9} \\
\hline
0.4 &     0.1643886586474064... \times 10^{-8} &   0.1643886612493694... \times 10^{-8} \\
 \hline
\end{array}
$$
\caption{Comparison of the bound (\ref{3.9}) with the true numerical error obtained using (\ref{3.7}) to approximate 
(\ref{3.5})  for 
different values of $z$.  In the calculations,
the values of $\nu$, $m$ and $r$  are fixed  ($\nu=100$, $m=1$, $r=4$). }
\end{table}

\begin{table}
\label{table2}
$$
\begin{array}{lll}
\hline
z     & \mbox{True error} & \mbox{Error bound}   \\
\hline
0.01 & 0.6400062383792656...  \times 10^{-8} & 0.6400062383792815...     \times 10^{-8}\\
\hline
0.05 & 0.7935186942078208...  \times 10^{-8} & 0.7935186942078481...     \times 10^{-8}\\
\hline
0.1      & 0.9091671462746869... \times 10^{-8} & 0.9091671462747565...   \times 10^{-8}\\
\hline
0.15    & 0.9835358907013243... \times 10^{-8} & 0.9835358907015595...    \times 10^{-8}\\
\hline
0.2      & 0.1040696060554677... \times 10^{-7} & 0.1040696060555794...    \times 10^{-7}\\
\hline
0.25    & 0.1097733448582064... \times 10^{-7} & 0.1097733448588827...   \times 10^{-7}\\
\hline 
0.3,      &0.1168879400116339... \times 10^{-7} & 0.1168879400162518...  \times 10^{-7}\\
\hline
0.35,    &0.1268224650280313... \times 10^{-7} & 0.1268224650616587...   \times 10^{-7}\\
\hline
0.4,      &0.1412373422298345... \times 10^{-7} & 0.1412373424897128...    \times 10^{-7}\\
 \hline
\end{array}
$$
\caption{Comparison of the bound (\ref{3.10}) with the true numerical error obtained using (\ref{3.8}) to approximate (\ref{3.6})
 for  different values of $z$.  In the calculations,
the values of $\nu$, $m$ and $r$  are fixed  ($\nu=100$, $m=1$, $r=4$). }
\end{table}

As an illustration of the high accuracy of the bounds, in Figure \ref{fig:fign1} we show the relative errors in the approximation of the true errors by the error bounds. The relative error ($e_r$) is given by 

\begin{equation}
\label{reler}
e_r=\left|1-\displaystyle\frac{\mbox{True error}}{\mbox{Error bound}}\right|\,.
\end{equation}

\begin{figure}[htbp]
  \centering
  \includegraphics[width=13cm]{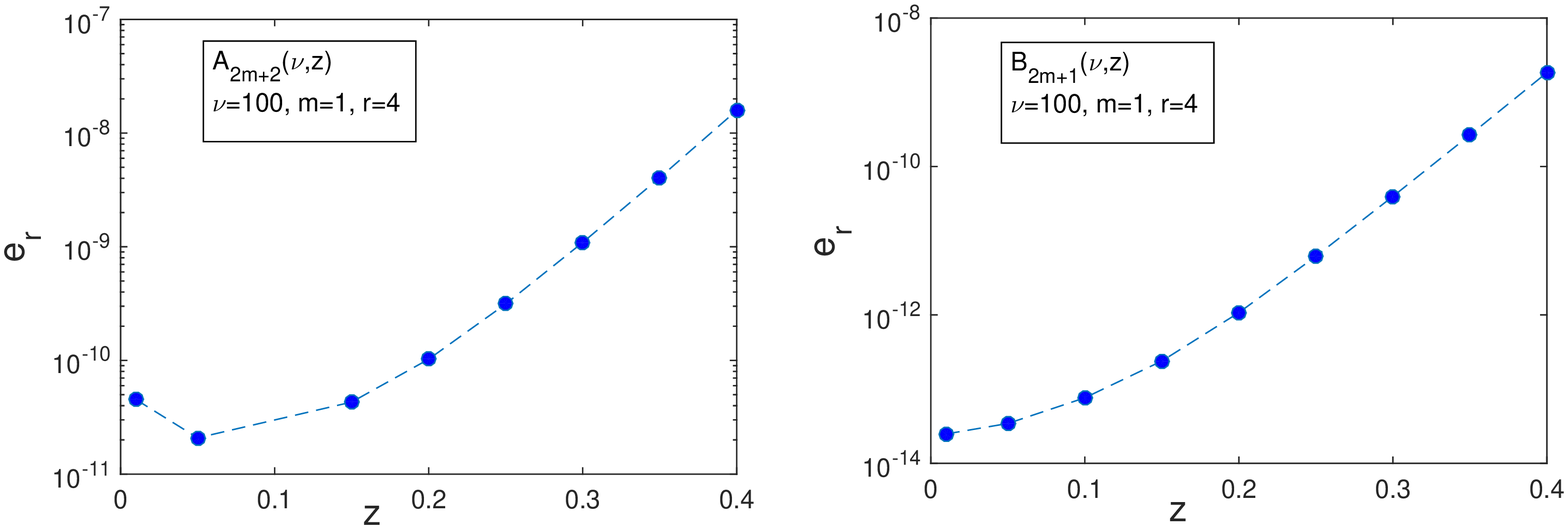}
  \caption{Relative errors in the comparison of the error bounds given in (\ref{3.9}) and (\ref{3.10}) with the true numerical errors obtained 
using (\ref{3.7}) and (\ref{3.8}), respectively, 
for fixed values of $\nu$, $m$ and $r$. }
  \label{fig:fign1}
\end{figure}

A second test of the error bounds (\ref{3.9}) and (\ref{3.10}), is shown in Tables \ref{table3} and \ref{table4}, respectively. 
As before, the bounds are compared with the true numerical errors obtained using (\ref{3.7}) and (\ref{3.8}) 
to approximate (\ref{3.5}) and (\ref{3.6}).
In the calculations,
the values of $z$, $m$ and $r$  are fixed  ($z=0.2$, $m=1$, $r=4$) and few values of $\nu$ have been considered 
for comparison. 

\begin{table}
\label{table3}
$$
\begin{array}{lllll}
\hline
\nu      & \mbox{True error} & \mbox{Error bound} & e_r   \\
  \hline
10   &    0.108936723... \times 10^{-5}     &   0.129960479... \times 10^{-5}     &  0.16     \\    
\hline
20   &    0.680802433... \times 10^{-7}     &    0.680830461  ... \times 10^{-7}     &  0.41 \times 10^{-4}     \\
 \hline
30   &    0.134477719... \times 10^{-7}       &    0.134477933... \times 10^{-7}     &  0.16 \times 10^{-5}     \\    
\hline
40   &    0.425493976... \times 10^{-8}                 & 0.425494044     ... \times 10^{-8}     &  0.15 \times 10^{-5}     \\
 \hline
50   &    0.174281968  ... \times 10^{-8}     &           0.174281973 ... \times 10^{-8}     &  0.26 \times 10^{-7}     \\
 \hline
\end{array}
$$
\caption{Comparison of the bound (\ref{3.9}) with the true numerical error obtained using (\ref{3.7}) to approximate 
(\ref{3.5})  for $z=0.2$,
$m=1$ and $r=4$ and different values of $\nu$. The relative errors given in Eq.(\ref{reler}) are shown in the last column. }
\end{table}

\begin{table}
\label{table4}
$$
\begin{array}{lllll}
\hline
\nu      & \mbox{True error} & \mbox{Error bound} & e_r   \\
  \hline
10   &      0.103831850 ... \times 10^{-4}     &     0.103949514 ... \times 10^{-4}     &  0.10 \times 10^{-2}     \\    
\hline
20   &      0.130014633... \times 10^{-5}                 &  0.130014912 ... \times 10^{-5}     &  0.21 \times 10^{-5}     \\
 \hline
30   &      0.385352565... \times 10^{-6}     &          0.385352587... \times 10^{-6}     &  0.55 \times 10^{-7}     \\    
\hline
40   &      0.16258894713... \times 10^{-6}     &  0.16258894780 ... \times 10^{-6}     &  0.40 \times 10^{-8}     \\
 \hline
50   &      0.83249887531  ... \times 10^{-7}     &  0.83249887577 ... \times 10^{-7}     &  0.55 \times 10^{-9}     \\
 \hline
\end{array}
$$
\caption{Comparison of the bound (\ref{3.10}) with the true numerical error obtained using (\ref{3.8}) to approximate 
(\ref{3.6})  for $z=0.2$,
$m=1$ and $r=4$. The relative errors given in Eq.(\ref{reler}) are shown in the last column.   }
\end{table}

\subsection{Computations close to the turning point}

The hardest part in computing $d_{n}(\nu)$ and $\tilde{d}_{n}(\nu)$ (given by (\ref{d2m+1}) and (\ref{d2m+2}) with $u$ replaced by $\nu$) are the terms $\omega _{n}(\nu )$ and $\varpi _{n}(\nu) $, which are given by (\ref{eq68}) and (\ref{eq69}) respectively, with $u$ replaced by $\nu$. We do not have to compute these to maximum accuracy, rather simple bounds will suffice.

With this in mind we start by using the Cauchy-Schwarz inequality in (\ref{eq68}), to obtain 
\begin{multline}
\int_{\gamma _{j,l}}\left\vert
\hat{F}_{k}(t) \hat{F}_{s+n-k-1}(t)
f^{1/2}(t)dt\right\vert \\ \leq \left\{ \int_{\gamma _{j,l}}{\left\vert 
\left\{\hat{F}_{k}(t)\right\}^{2}
f^{1/2}(t)dt\right\vert }\right\} ^{1/2}\left\{ 
\int_{\gamma_{j,l}}{\left\vert
\left\{\hat{F}_{s+n-k-1}(t)\right\}^{2}
f^{1/2}(t)dt
\right\vert }\right\} ^{1/2}.
\label{CSI}
\end{multline}

We require the maxima of these and the other integrals in (\ref{eq68}) and (\ref{eq69}) over all six paths $\gamma _{j,l}$, but we can simplify as follows. Firstly by Schwarz's symmetry principle we only need to consider $\Im t\geq 0$. Next all points on the upper part of $\Gamma $ ($\Gamma ^{+}$ say) can be accessed by a progressive path which is part, or the whole, of a path that consists of the union of

(i) a line from $t=0$ to $t=r_{0}$, or from $t=1+ir_{0}$ to $t=1+i\infty $, and

(ii) a quarter circle $t=1+r_{0}e^{i\phi }$ from $\phi =0$ to $\phi =\frac{1}{2}\pi $, or from $\phi =\frac{1}{2}\pi $ to $\phi =\pi $.

Therefore for each $k=1,2,3,\cdots $ and $m=1,2$ 
\begin{equation}
\int_{\gamma _{j,l}}\left\vert \left\{\hat{F}_{k}(t)\right\}^{m}f^{1/2}(t)dt
\right\vert \leq F_{m,k},
\label{n1}
\end{equation}
where (using the obvious parametrization for each sub-path)
\begin{multline}
F_{m,k}=\max\left\{\int_{0}^{1-r_{0}}\left\vert \left\{\hat{F}_{k}(t)\right\}^{m}f^{1/2}(t)\right\vert dt,\int_{r_{0}}^{\infty }\left\vert 
\left\{\hat{F}_{k}(1+is)\right\}^{m} f^{1/2}\left( 1+is\right) \right\vert
ds\right\}  \\
+r_{0}\max 
\left\{ \int_{0}^{\pi /2}\left\vert 
\left\{\hat{F}_{k}\left( 1+r_{0}e^{i\phi }\right)\right\}^{m}
{f^{1/2}\left( 1+r_{0}e^{i\phi }\right) }
\right\vert d\phi, \right. 
\\ \left. 
\int_{\pi /2}^{\pi }
\left\vert 
\left\{\hat{F}_{k}\left( 1+r_{0}e^{i\phi }\right)\right\}^{m} 
f^{1/2}\left( 1+r_{0}e^{i\phi }\right) 
\right\vert d\phi 
\right\}.  
\label{n2}
\end{multline}

Thus from (\ref{eq68}), (\ref{eq69}) and (\ref{CSI})
\begin{equation}
\omega _{n}(\nu )\leq \boldsymbol{\omega}_{n}(\nu ):=2F_{1,n}
+\sum\limits_{s=1}^{n-1}\dfrac{1}{\nu ^{s}}{\ \sum\limits_{k=s}^{n-1}}
\left\{ F_{2,k}\right\} ^{1/2}\left\{ F_{2,s+n-k-1}\right\}^{1/2},
\label{n3}
\end{equation}
and
\begin{equation}
\varpi _{n}(\nu )\leq \boldsymbol{\varpi}_{n}(\nu ):=4\sum\limits_{s=0}^{n-2}
\frac{F_{1,s+1}}{\nu ^{s}}.
\label{n4}
\end{equation}

Next we make the interval of integration finite for the second integral of (\ref{n2}) by the following simple change of variable
\begin{multline}
\int_{r_{0}}^{\infty }\left\vert 
\left\{\hat{F}_{k}(1+is)\right\}^{m}
f^{1/2}\left( 1+is\right) \right\vert ds \\ =\int_{0}^{1/r_{0}}{\left\vert 
\left\{\hat{F}_{k}\left(1+it^{-1}\right)\right\}^{m}
f^{1/2}\left( 1+it^{-1}\right)
\right\vert }t^{-2} dt.
\label{n5}
\end{multline}
Here integrand of the second integral is $\mathcal{O}\left(t^{m k+m-2}\right)$ as $t\rightarrow 0$, and hence is straightforward to compute numerically.

We take the optimal choice $r_{0}=1$. Then the first integral in (\ref{n2}) vanishes, and (for example) for $k=10$ and $m=1$ we find for the other three
\begin{equation}
\int_{0}^{1}{\left\vert {\hat{F}}_{10}\left( {1+it^{-1}}\right) {
f^{1/2}\left( 1+it^{-1}\right) }\right\vert }t^{-2}dt=445.18\cdots,
\label{n6}
\end{equation}
\begin{equation}
{\int_{0}^{\pi /2}{\left\vert {{\hat{F}_{10}}}\left( 1+e^{i\phi }\right) {
f^{1/2}\left( 1+e^{i\phi }\right) }\right\vert d\phi =}}3.10\cdots {\times 10
}^{4},
\label{n7}
\end{equation}
\begin{equation}
\int_{\pi /2}^{\pi }\left\vert \hat{F}_{10}\left( 1+e^{i\phi
}\right) {f^{1/2}\left( 1+e^{i\phi }\right) }\right\vert d\phi =1.79\cdots 
\times 10^{3},
\label{n8}
\end{equation}
and hence from (\ref{n2}) and (\ref{n5}) $F_{1,10}=3.15\cdots
\times 10^{4}$.

Similarly, for $m=2$ we obtain
\begin{equation}
\int_{0}^{1}\left\vert 
\left\{\hat{F}_{10}\left( 1+it^{-1}\right)\right\}^{2}
{f^{1/2}\left( 1+it^{-1}\right) }\right\vert t^{-2}dt=1.14\cdots \times 10^{6}  
\label{n11}
\end{equation}
\begin{equation}
\int_{0}^{\pi /2}\left\vert 
\left\{\hat{F}_{10}\left( 1+e^{i\phi }\right)\right\}^{2}
f^{1/2}\left( 1+e^{i\phi }\right) \right\vert d\phi =8.06\cdots \times 10^{8} 
\label{n9}
\end{equation}
\begin{equation}
\int_{\pi /2}^{\pi }\left\vert
\left\{\hat{F}_{10}\left( 1+e^{i\phi }\right)\right\}^{2}
f^{1/2}\left( 1+e^{i\phi }\right) \right\vert d\phi =5.17\cdots 
\times 10^{6}  
\label{n10}
\end{equation}
and hence from (\ref{n2}) and (\ref{n5}) $\left\{ F_{2,10}\right\}^{1/2}=2.84\cdots \times 10^{4}$.

Note in computing some of these integrals by trapezoidal approximation in Maple it is quicker to include in one or both of the limits a decimal number ($0.5\pi$ instead of $\pi/2$, etc.). This prevents Maple attempting to find exact values at the nodes (which can result in a very slow computation).

For other terms needing to be computed in (\ref{d2m+1}) and the bound (\ref{3.15d}) we obtain (again for $r_{0}=1$) 
\begin{equation}
\Upsilon =\underset{z\in \Gamma ^{+}} \inf \left\vert \zeta f(z)\right\vert ^{1/4}=0.935\cdots, 
\label{n12}
\end{equation}
and
\begin{equation}
\rho =\underset{z\in \Gamma ^{+}}{\inf }\left\vert \xi \right\vert
=0.685\cdots,
\label{n14}
\end{equation}
where we used $\Gamma ^{+}$ in place of $\Gamma $, again by virtue of Schwarz's symmetry principle.

To simplify further we replace $M_{s}$, $N_{s}$ in (\ref{eq96}) by a common upper bound $E_{s}$ given by (\ref{G7}) (but again these suprema need only be taken over $\Gamma^{+}$). These, along with $F_{m,k}$, are one-time computations and can be stored.

We then have from (\ref{d2m+1}) 
\begin{equation}
d_{n}(\nu )\leq \boldsymbol{d}_{n}(\nu ):=2\exp \left\{ \sum\limits_{s=1}^{n-1}
\ \dfrac{E_{s}}{\nu ^{s}}\right\} \boldsymbol{e}_{n}(\nu )\left\{ 1+
\dfrac{\boldsymbol{e}_{n}(\nu )}{2\nu ^{n}}\right\} ^{2},
\label{n17}
\end{equation}
where 
\begin{multline}
\boldsymbol{e}_{n}(\nu) =\nu ^{n} \delta_{n}(\nu )+\boldsymbol{\omega}_{n}(\nu )\exp \left\{ \nu ^{-1}\boldsymbol{\varpi}_{n}(\nu )+\nu ^{-n}\boldsymbol{
\omega}_{n}(\nu )\right\}  \\
+\gamma _{n}(\nu ,\rho )\exp \left\{ \nu ^{-1}\beta _{n}(\nu ,\rho)+\nu ^{-n}\gamma _{n}(\nu ,\rho )\right\}.
\label{n18}
\end{multline}

Note that in the bound (\ref{3.15d}) $n=2m+2r+2$, and in (\ref{n18}) we then use
\begin{equation}
{\delta }_{2m+2r+2}(\nu )=\left( \frac{{1}}{2\pi }\right) ^{{1/2}}\frac{e^{
{\nu }}\Gamma (\nu )}{{\nu }^{{\nu -(1/2)}}}\exp \left\{ -{
\sum\limits_{j=0}^{m+r}{\frac{{C}_{2j+1}}{{\nu }^{2j+1}}}}\right\} -1,
\label{n19}
\end{equation}
since ${\delta }_{2m+2r+2,-1}(\nu )=$ ${\delta }_{2m+2r+2,1}(\nu )$. To
compute this number in a stable manner a few terms of an asymptotic expansion are used (see \cref{remark2}).

The term $\tilde{d}_{2m+2r+2}(\nu)$ is similarly bounded, and in this and (\ref{3.15b}) we use  $\boldsymbol{\omega}_{n}(\nu )$ and $\boldsymbol{\varpi}_{n}(\nu )$ as above, $\tilde{M}_{s}$, $\tilde{N}_{s}$ $\leq \tilde{E}_{s}$ where $\tilde{E}_{s}$ is given by (\ref{G17}) (with $\Gamma$ replaced by $\Gamma ^{+}$), and
\begin{equation}
\tilde{\Upsilon}=\underset{z\in \Gamma ^{+}} \sup \left\vert \zeta/f(z)\right\vert ^{1/4}=1.079\cdots.
\label{n21}
\end{equation}

\begin{table} [H]
\label{table5}
$$
\begin{array}{lllll}
\hline
\alpha, \,\nu      & \mbox{True error} & \mbox{Error bound} & e_r   \\
  \hline
0,\,\nu=50   &       0.23466561...       \times 10^{-9}  &  0.23479628...  \times 10^{-9}     &   0.55   \times 10^{-3}     \\
0,\,\nu=100     &   0.14668477...     \times 10^{-10}   &  0.14669749... \times 10^{-10}     &  0.86  \times 10^{-4}         \\           
\hline
\pi/6,\,\nu=50   & 0.23620960...  \times 10^{-9}        &  0.23634124...  \times 10^{-9}     &     0.55 \times 10^{-3}     \\
\pi/6,\,\nu=100&  0.14764999...  \times  10^{-10}      &  0.14766286...  \times 10^{-10}      &     0.87 \times 10^{-4}           \\           
\hline
\pi/3,\,\nu=50   & 0.24012769...   \times 10^{-9}     &    0.24026186... \times 10^{-9}     & 0.56  \times 10^{-3}     \\
\pi/3,\,\nu=100 & 0.15009938...  \times 10^{-10}       &  0.15011265... \times 10^{-10}   &   0.88  \times 10^{-4}         \\           
\hline
\pi/2,\,\nu=50   & 0.24448477... \times 10^{-9}     &       0.24462190... \times 10^{-9}     &  0.56 \times 10^{-3}     \\
\pi/2,\,\nu=100 & 0.15282325...  \times 10^{-10}            &       0.15283698...\times 10^{-10}    &   0.89  \times 10^{-4}          \\           
\hline
2\pi/3,\,\nu=50 &  0.24709299...  \times 10^{-9}     &     0.24723207...   \times 10^{-9}     & 0.56 \times 10^{-3}     \\
2\pi/3,\,\nu=100& 0.15445384...   \times 10^{-10}       & 0.15446788...  \times 10^{-10}     &  0.9 \times 10^{-4}             \\         
\hline
5\pi/6,\,\nu=50  & 0.24736359...   \times 10^{-9}     &    0.24750308...    \times 10^{-9}     &  0.56 \times 10^{-3}     \\
5\pi/6,\,\nu=100& 0.15462306... \times   10^{-10}     &   0.15463717...  \times 10^{-10}     & 0.91  \times 10^{-4}             \\     
\hline
\pi,\,\nu=50     &  0.24700694... \times 10^{-9}     &       0.24714629...   \times 10^{-9}     & 0.56 \times 10^{-3}     \\
\pi,\,\nu=100   &  0.15440013... \times  10^{-10}        &  0.15441421... \times 10^{-10}     &   0.91 \times 10^{-4}             \\ 

\hline    
\end{array}
$$
\caption{Comparison of the bound given in (\ref{3.15b}) with the true numerical error obtained using (\ref{3.15a})  
to approximate (\ref{3.5}). The relative errors given in Eq.(\ref{reler}) are shown in the last column.  }
\end{table}

\begin{table} [H]
\label{table6}
$$
\begin{array}{lllll}
\hline
\alpha, \,\nu      & \mbox{True error} & \mbox{Error bound} & e_r \\
  \hline
0,\,\nu=50   &      0.15449776... \times 10^{-7}    &  0.15455270... \times 10^{-7}     & 0.35   \times 10^{-3}     \\
0,\,\nu=100     &  0.19314772...\times 10^{-8}      &  0.19316474... \times 10^{-8}     &   0.88   \times 10^{-4}         \\           
\hline
\pi/6,\,\nu=50   & 0.15663409... \times 10^{-7}     &  0.15669006... \times 10^{-7}     & 0.36   \times 10^{-3}     \\
\pi/6,\,\nu=100 & 0.19581861...\times 10^{-8}       & 0.19583595... \times 10^{-8}      &   0.89    \times 10^{-4}           \\           
\hline
\pi/3,\,\nu=50   & 0.16258667... \times 10^{-7}     &  0.16264555... \times 10^{-7}     & 0.36  \times 10^{-3}     \\
\pi/3,\,\nu=100 & 0.20326069...\times 10^{-8}       & 0.20327894... \times 10^{-8}   & 0.89    \times 10^{-4}         \\           
\hline
\pi/2,\,\nu=50  &  0.17094997... \times 10^{-7}     &  0.17101298... \times 10^{-7}     & 0.37 \times 10^{-3}     \\
\pi/2,\,\nu=100 & 0.21371675... \times 10^{-8}      & 0.21373629... \times 10^{-8}    &    0.91  \times 10^{-4}          \\           
\hline
2\pi/3,\,\nu=50  & 0.17947572... \times 10^{-7}     & 0.17954297... \times 10^{-7}     & 0.37  \times 10^{-3}     \\
2\pi/3,\,\nu=100& 0.22437593...\times 10^{-8}      & 0.22439679... \times 10^{-8}     &   0.93 \times 10^{-4}             \\           
\hline
5\pi/6,\,\nu=50 &  0.18571359... \times 10^{-7}     & 0.18578395... \times 10^{-7}     &   0.38 \times 10^{-3}     \\
5\pi/6,\,\nu=100& 0.23217472...\times 10^{-8}     &  0.23219656... \times 10^{-8}     &  0.94  \times 10^{-4}             \\           
\hline
\pi,\,\nu=50   &    0.18797216... \times 10^{-7}     & 0.18804363... \times 10^{-7}     &   0.38 \times 10^{-3}     \\
\pi,\,\nu=100     & 0.23499846...\times 10^{-8}     &  0.23502064... \times 10^{-8}     &   0.94   \times 10^{-4}             \\ 
\hline    
\end{array}
$$
\caption{ Comparison of the bound given in (\ref{3.15d}) with the true numerical error obtained using (\ref{3.15c})  
to approximate (\ref{3.6}). The relative errors given in Eq.(\ref{reler}) are shown in the last column.  }
\end{table}

In Table \ref{table5} we show the accuracy of the error bounds given in (\ref{3.15b}) 
for $m=1,$, $r=4$ and two values of $\nu$ ($\nu=50,\,100$). The
values of the argument $z$ considered are $z=1+0.1e^{i\alpha}$, for different values of $\alpha$.

In Table \ref{table6}, results for the
  accuracy of the error bounds given in (\ref{3.15d}) are shown. For computing the integrals in (\ref{3.15a}), 
(\ref{3.15c}) and the coefficients ${\cal{G}}_{m,n}^*(z)$ and  ${\cal{\tilde{G}}}_{m,n}^*(z)$,
we use a circular path of integration enclosing the turning point with parametrization $t(\theta)=z_c+Re^{i\theta}$;
we use $\theta \in (0,\,2\pi)$ with $z_c=1.5$ and $R=1.3$.  
The resulting integrals are well approximated using the trapezoidal rule with $500$ points over the contour. See \cite{Bornemann:2011:AAS} for details of the efficacy of this numerical method for evaluating Cauchy integrals.

\begin{table} [h]
\label{table7}
$$
\begin{array}{lllll}
\hline
R_z      & \mbox{True Error} & \mbox{Bound} & e_r  \\
  \hline
10^{-3},\,(\ref{3.15b})  &  0.15170987... \times 10^{-10}    & 0.15172344... \times 10^{-10}     & 0.89   \times 10^{-4}     \\
10^{-3},\,(\ref{3.15d})   & 0.21359742...\times 10^{-8}      & 0.21361694... \times 10^{-8}     & 0.91   \times 10^{-4}         \\           
\hline
10^{-2},\,(\ref{3.15b})   & 0.15206221... \times 10^{-10}   & 0.15207584... \times 10^{-10}     &   0.89 \times 10^{-4}     \\
10^{-2},\,(\ref{3.15d}) &   0.21549252...\times 10^{-8}      & 0.21551227... \times 10^{-8}      &   0.92   \times 10^{-4}           \\           
\hline
0.1,\,(\ref{3.15b})  &         0.15440013... \times 10^{-10}   & 0.15441421... \times 10^{-10}     & 0.91 \times 10^{-4}     \\
0.1,\,(\ref{3.15d}) &          0.23499846...\times 10^{-8}      & 0.23502064... \times 10^{-8}   &   0.94 \times 10^{-4}         \\           
\hline
0.2,\,(\ref{3.15b})  &         0.15357545... \times 10^{-10}    & 0.15358950... \times 10^{-10}     &  0.91 \times 10^{-4}     \\
0.2,\,(\ref{3.15d}) &          0.25747685... \times 10^{-8}      & 0.25750187... \times 10^{-8}    &   0.97   \times 10^{-4}          \\           
\hline
0.3,\,(\ref{3.15b})  &         0.14740882... \times 10^{-10}     &0.14742194... \times 10^{-10}     & 0.89 \times 10^{-4}     \\
0.3,\,(\ref{3.15d})&           0.27983939... \times 10^{-8}      & 0.27986717... \times 10^{-8}     &  0.99 \times 10^{-4}             \\           
\hline
0.4,\,(\ref{3.15b}) &          0.13356389... \times 10^{-10}     &0.13357483... \times 10^{-10}     &  0.82  \times 10^{-4}     \\
0.4,\,(\ref{3.15d})&           0.30030770...\times 10^{-8}     &   0.30033781... \times 10^{-8}     & 0.10  \times 10^{-3}             \\           
\hline
0.5,\, (\ref{3.15b})   &       0.10925753... \times 10^{-10}     & 0.10926467... \times 10^{-10}     &   0.65 \times 10^{-4}     \\
0.5,\,(\ref{3.15d})    &      0.31580277...\times 10^{-8}     &    0.31583420... \times 10^{-8}     &   0.99   \times 10^{-4}             \\ 
\hline    
\end{array}
$$
\caption{ Comparison of the bounds given in (\ref{3.15b}) and (\ref{3.15d}) 
with the true numerical errors obtained using (\ref{3.15a}) and (\ref{3.15c})  
to approximate (\ref{3.5}) and (\ref{3.6}), respectively. The values of the argument $z$ are varied ($z=1-R_z$). $\nu$, $m$ and $r$ 
are fixed in the calculations ($\nu=100$, $m=1$, $r=4$).     The relative errors in the
comparisons given in Eq.(\ref{reler}) are shown in the last column. }
\end{table}

Another test of the accuracy of the bounds is shown in Table \ref{table7}. The comparison of the bounds with the true numerical errors is shown for different values of $z$. The values of $\nu$, $m$ and $r$ are fixed in the calculations ($\nu=100$, $m=1$, $r=4$).

\appendix
\section{Proof of \cref{thm:as}}
\label{secA} 
We begin with the following.
\begin{proposition}
Let 
\begin{equation} S_{n}=\sum\limits_{j=1}^{n-1}\frac{{j!(n-j)!}}{(n-1)!}=n\sum \limits_{j=1}^{n-1}\dbinom{n}{j}^{-1} 
\label{Sn}.
\end{equation}
Then 
\begin{equation}
S_{n}\leq \frac{2{n}}{n+2}\left\{ 1+\frac{32}{n+3}\right\} +\frac{{3n}\left( n+1\right) }{2}\left( \frac{3}{4}\right) ^{n}\ \left( n=2,3,4,\cdots \right),
\label{Snbound} 
\end{equation}
and moreover 
\begin{equation}
S_{n}=2+\mathcal{O}\left( \frac{1}{n}\right) \ \left( n\rightarrow \infty \right).
\label{SnO} 
\end{equation}

\end{proposition}

\begin{remark}
From (\ref{SnO}) we observe that the bound (\ref{Snbound}) is asymptotically sharp for large $n$.
\end{remark} 

\begin{proof}
Converting the factorials in (\ref{Sn}) to the Gamma function we obtain (see also \cite{Sury:2004:IIR}) 
\begin{equation}
S_{n}=n\left( n+1\right) \sum\limits_{j=1}^{n-1}\frac{\Gamma \left( j+1\right) \Gamma {(n-j+1)}}{\Gamma \left( n+2\right) }=n\left( n+1\right) \sum\limits_{j=1}^{n-1}B\left( j+1,{n-j+1}\right),
\label{eq22} 
\end{equation}
where $B\left( p,{q}\right) $ is the Beta function 
\begin{equation}
B\left( p,{q}\right) =\frac{\Gamma(p) \Gamma(q)}{ \Gamma(p+q) }=\int_{0}^{1}t^{p-1}(1-t)^{q-1}dt
\quad (\Re(p)>0, \, \Re(q)>0).
\label{beta} 
\end{equation}
Therefore 
\begin{multline}
 S_{n} ={n}\left( n+1\right) \sum\limits_{j=1}^{n-1}\int_{0}^{1}t^{j}(1-t)^{n-j}dt={n}\left( n+1\right) \int_{0}^{1}(1-t)^{n}\sum\limits_{j=1}^{n-1}\left( \frac{t}{1-t}\right) ^{j}dt 
\label{eq23} \\ ={n}\left( n+1\right) \int_{0}^{1}\frac{1}{1-2t}\left\{ t(1-t)^{n}-t^{n}(1-t)\right\} dt
\\={2n}\left( n+1\right) \int_{0}^{1/2}\frac{1 }{1-2t}\left\{ t(1-t)^{n}-t^{n}(1-t)\right\} dt,  
\end{multline}
the last integral coming from symmetry of the integrand about $t=\frac{1}{2}$. It follows that
\begin{equation}
S_{n}={2n}\left( n+1\right) \int_{0}^{1/4}\frac{t(1-t)^{n}}{1-2t}dt+R_{n},
\label{eq30} 
\end{equation}
where 
\begin{multline}
R_{n}={2n}\left( n+1\right) \int_{1/4}^{1/2}\frac{1}{1-2t}\left\{ t(1-t)^{n}-t^{n}(1-t)\right\} dt
\\ -{2n}\left( n+1\right) \int_{0}^{1/4}\frac{ t^{n}(1-t)}{1-2t}dt.
\label{eq31} 
\end{multline}
Next by the maximum modulus theorem and the triangle inequality
\begin{multline}
 \sup_{t\in \left( \frac{1}{4},\frac{1}{2}\right) }\left\vert \frac{1}{1-2t} \left\{ t(1-t)^{n}-t^{n}(1-t)\right\} \right\vert
 \\ \leq \sup_{\left\vert z-(1/2)\right\vert =1/4}\frac{1}{2\left\vert z-\frac{1}{2}\right\vert } \left\vert z(1-z)^{n}-z^{n}(1-z)\right\vert  \\ \leq \sup_{\left\vert z-(1/2)\right\vert =1/4}2\left\{ \left\vert z\right\vert \left\vert 1-z\right\vert ^{n}+\left\vert z\right\vert ^{n}\left\vert 1-z\right\vert \right\} \leq 4\left( \frac{3}{4}\right) ^{n+1}, 
\label{eq32} 
\end{multline}
 since $\left\vert z\right\vert $ and $\left\vert 1-z\right\vert $ both have suprema of $\frac{3}{4}$ on this circle. Thus from (\ref{eq31})
\begin{equation}
R_{n}\leq {8n}\left( n+1\right) \left( \frac{3}{4}\right) ^{n+1}\int_{1/4}^{1/2}dt=\frac{{3n}\left( n+1\right) }{2}\left( \frac{3}{4} \right) ^{n},
\end{equation}
and hence from (\ref{eq30})
\begin{equation}
S_{n}\leq {2n}\left( n+1\right) \int_{0}^{1/4}\frac{t(1-t)^{n}}{1-2t}dt+ \frac{{3n}\left( n+1\right) }{2}\left( \frac{3}{4}\right) ^{n}.
\label{eq33} 
\end{equation}

We now integrate by parts three times, and retain only positive terms, and deduce that 
\begin{multline}
 \int_{0}^{1/4}\frac{t(1-t)^{n}}{1-2t}dt \leq \frac{1}{\left( n+1\right) \left( n+2\right) } \\ +\frac{4}{\left( n+1\right) \left( n+2\right) \left( n+3\right) }\left\{ 1+6\int_{0}^{1/4}\frac{(1-t)^{n+3}}{\left( 1-2t\right) ^{4}}dt\right\}.
\label{Rnbound} 
\end{multline}
 The integral on the RHS is $\mathcal{O}( n^{-1}) $ for large $n$, but the following simple bound will suffice:
\begin{equation} \int_{0}^{1/4}\frac{(1-t)^{n+3}}{(1-2t) ^{4}}dt\leq \int_{0}^{1/4}\frac{1}{(1-2t) ^{4}}dt=\frac{7}{6}.
\label{eq34} 
\end{equation}
The desired bound (\ref{Snbound}) follows from (\ref{eq33}) - (\ref{eq34}). Finally, the discarded negative terms in establishing (\ref{Rnbound}) are ${\mathcal{O }}(n^{-1}) $ and hence (\ref{SnO}) follows. 
\end{proof}

Let us establish (\ref{eq19}). To this end define $a_{s}=\frac{5}{36}2^{-s}c_{s}$. Then $c_{1}=1$ and $c_{2}=2$, with subsequent terms given by 
\begin{equation}
c_{s+1}=\left( {s+1}\right) c_{s}+\frac{5}{36}\sum\limits_{j=1}^{s-1}{ c_{j}c_{s-j}}.
\label{eq101} 
\end{equation}
To get to (\ref{eq19}) wish to prove for some $A \geq 1$ that
\begin{equation} 
\label{eq101c}
s!\leq c_{s}\leq A^{s}s!.
\end{equation}
The lower bound is simple to establish by induction (starting with $c_1=1!$), since it is obvious from $c_{1}$ and $c_{2}$\ both being positive, and the recursion (\ref{eq101}), that $c_{s}\geq 0$ for all $s$. Hence assuming $c_{s}\geq s!$ we have from (\ref{eq101})
\begin{equation}
c_{s+1}\geq \left( s{+1}\right) c_{s}\geq (s+1) s!=(s+1)! .
\end{equation}

Next assume the upper bound in (\ref{eq101c}) holds for $c_{j}$ ($j=1,2,3,\cdots ,s$). Then from (\ref{eq101})
\begin{equation}
\label{eq101d}
c_{s+1}\leq A^{s}(s+1) !+{\frac{5}{36}A^{s}} \sum\limits_{j=1}^{s-1}{j!}\left( {s-j}\right) ! .
\end{equation}
Now from (\ref{Sn}) 
\begin{equation} {\frac{5}{36(s-1)!}}\sum\limits_{j=1}^{s-1}{j!(s-j)!}
\leq K_{s},
\end{equation}
where
\begin{equation} 
K_{s}=\frac{5{s}}{ 18(s+2) }\left\{ 1+\frac{32}{s+3}\right\} +{\frac{5}{24}s} (s+1) \left( \frac{3}{4}\right) ^{s}.
\end{equation}
Thus from (\ref{eq101d}) we have
\begin{equation}
c_{s+1}\leq A^{s}(s+1)!+A^{s}K_{s}(s-1)! .
\end{equation}
So in order for (\ref{eq101c}) to be true for $s$ replaced by $s+1$ it is sufficient for 
\begin{equation}
A^{s}(s+1)!+{A}^{s}K_{s}(s-1)!\leq A^{s+1}(s+1)! ,
\end{equation}
for all $s$, or equivalently
\begin{equation}
\label{eq101e}
A \geq 1+\frac{K_{s}}{s(s+1) }.
\end{equation}
Now since $c_{1}=1$ and $c_{2}=2$ we see that (\ref{eq101c}) certainly holds for any $A \geq 1$. Thus (\ref{eq101e}) holds for all other values of $s$ if we choose
\begin{equation}
A=1+\frac{K_{2}}{2\left( 2+1\right) }=\frac{4453}{3456},
\end{equation}
since the RHS of (\ref{eq101e}) is a decreasing function of $s$. Our asserted upper bound of (\ref{eq19}) is then established by inserting this value of $A$ into (\ref{eq101c}) and recalling that $a_{s}=\frac{5}{36}2^{-s}c_{s}$.

Next we consider proving (\ref{eq19a}), and to do so we let $\tilde{a}_{s}=-2^{-s}\tfrac{7}{36}\tilde{c}_{s}$. We wish to show that $ (s-1)!\leq \tilde{c}_{s}\leq s!$, which is clearly true for the first two terms since $\tilde{c}_{1}=1$ and $\tilde{c}_{2}=2$. Now from (\ref{eq101c}) we have 
\begin{equation}
\tilde{c}_{s+1}=\left( {s+1}\right) \tilde{c}_{s}-{\frac{7}{36}} \sum\limits_{j=1}^{s-1}{\tilde{c}_{j}\tilde{c}_{s-j}}. 
\label{eq20} 
\end{equation}

Again we proceed by using induction. Assume the hypotheses are true for $\tilde{c}_{j}$ ($j=1,2,3,\cdots,s$). Then $\tilde{c}_{s}\leq s!$ and $\tilde{c}_{j}\tilde{c}_{s-j}\geq (j-1)!(s-j-1)!>0$ ($j=1,2,\cdots s-1$), and hence 
\begin{equation}
\tilde{c}_{s+1}\leq (s+1) s!-{\frac{7}{36}} \sum\limits_{j=1}^{s-1}\tilde{c}_{j}\tilde{c}_{s-j}<\left( s+1\right) ! 
\label{eq20a} 
\end{equation}
as required for the upper bound.

It remains to prove the lower bound of (\ref{eq19a}). Now it can be verified numerically that$\ \tilde{c}_{s} \geq (s-1)!$ for $ s=1,2,3,\cdots ,24$. Consider any $s\geq 24$ and assume for $j=1,2,3,\cdots ,s$ that $\tilde{c}_{j}\geq (j-1)!$ (and hence of course $\tilde{c}_{j}\geq 0 $). Since we have established that $\tilde{c}_{s}\leq s!$ for all $s$ it follows from (\ref{eq20}) and the induction hypothesis that 
\begin{multline} 
\tilde{c}_{s+1}\geq (s+1) (s-1) !-\frac{7}{36} \sum\limits_{j=1}^{s-1}{\tilde{c}_{j}\tilde{c}_{s-j}}
\\ \geq s!+(s-1)!
 -{\frac{7}{36}}\sum\limits_{j=1}^{s-1}{j!(s-j)!\geq }s!+(s-1)!-\frac{7}{36} (s-1)!S_{s}. 
\label{eq21} 
\end{multline}
From (\ref{Snbound}) it is straightforward to show by explicit computation of that bound that 
\begin{equation} 
{\tfrac{7}{36}}S_{s}<1\text{ for }s\geq 24, 
\end{equation}
and thus from (\ref{eq21}) $\tilde{c}_{s+1}\geq s!$ for $s\geq 24$, as desired. 

\section*{Acknowledgments}
Financial support from Ministerio de Ciencia e Innovaci\'on, Spain, 
project PGC2018-098279-B-I00 (MCIU/AEI/FEDER, UE)
 is acknowledged. 

\bibliographystyle{siamplain}
\bibliography{biblio}

\end{document}

%% file: ex_shared.tex
\usepackage{lipsum}
\usepackage{float}
\usepackage{amsfonts}
\usepackage{graphicx}
\usepackage{epstopdf}
\usepackage{algorithmic}
\ifpdf
  \DeclareGraphicsExtensions{.eps,.pdf,.png,.jpg}
\else
  \DeclareGraphicsExtensions{.eps}
\fi

\newsiamremark{hypothesis}{Hypothesis}
\crefname{hypothesis}{Hypothesis}{Hypotheses}
\newsiamthm{claim}{Claim}
\newtheorem{remark}{Remark}
\nolinenumbers

\headers{Sharp error bounds for turning point expansions}{T. M. Dunster, A. Gil, and J. Segura}

\title{Sharp error bounds for turning point expansions}

\author{T. M. Dunster\thanks{Department of Mathematics and Statistics, San Diego State University, 5500 Campanile Drive, San Diego, CA 92182, USA. 
  (\email{mdunster@sdsu.edu}, \url{https://tmdunster.sdsu.edu}).}
\and A. Gil\thanks{Departamento de Matem\'atica Aplicada y CC. de la Computaci'on, ETSI Caminos, Universidad de Cantabria, 39005-Santander, Spain. 
  (\email{amparo.gil@unican.es}, \email{javier.segura@unican.es}).}
\and J. Segura\footnotemark[3]}

\usepackage{amsopn}

\makeatletter
\newcommand*{\addFileDependency}[1]{
  \typeout{(#1)}
  \@addtofilelist{#1}
  \IfFileExists{#1}{}{\typeout{No file #1.}}
}
\makeatother
